\documentclass[11pt]{amsart}
\usepackage{amssymb,amsmath,color,amsthm} 
\usepackage{booktabs,graphicx,fullpage}

\DeclareMathOperator{\asc}{asc}
\DeclareMathOperator{\Col}{\textsc{col}}
\DeclareMathOperator{\Hom}{Hom}

\DeclareMathOperator{\rev}{rev}
\DeclareMathOperator{\wt}{wt}

\newcommand{\ov}[1]{\overline{#1}}
\newcommand{\Q}{\mathbb{Q}}
\newcommand{\X}{f}
\newcommand{\Z}{\mathbb{Z}}

\newcommand{\mcH}{\mathcal{H}}
\newcommand{\mcS}{\mathcal{S}}
\newcommand{\Sym}{\mathfrak{S}}
\newcommand{\QSym}{\mathrm{QSym}}

\newcommand{\mt}{\widetilde{m}}
\newcommand{\bx}{\mathbf{x}}

\theoremstyle{plain}
\theoremstyle{definition}
\newtheorem{theorem}{Theorem}
\newtheorem{lemma}[theorem]{Lemma}
\newtheorem{prop}[theorem]{Proposition}

\newtheorem{conjecture}[theorem]{Conjecture}

\newtheorem{remark}[theorem]{Remark}
\newtheorem{example}[theorem]{Example}

\begin{document}

\title{A rooted variant of Stanley's chromatic~symmetric~function}

\date{\today}

\author{Nicholas A. Loehr}
\thanks{This work was supported by a grant from the Simons Foundation/SFARI
(\#633564, N.A.L.).  }
\address{Dept. of Mathematics, Virginia Tech, Blacksburg, VA 24061-0123 }
\email{nloehr@vt.edu}

\author{Gregory S. Warrington}
\thanks{This work was supported by a grant from the Simons Foundation/SFARI 
(\#429570, G.S.W.). }
\address{Dept. of Mathematics and Statistics, University of Vermont, Burlington, VT 05401}
\email{gregory.warrington@uvm.edu}


\begin{abstract}
Richard Stanley defined the chromatic symmetric function $X_G$ of a graph $G$
and asked whether there are non-isomorphic trees $T$ and $U$ with $X_T=X_U$.
We study variants of the chromatic symmetric function for rooted
graphs, where we require the root vertex to either use or avoid a specified
color. We present combinatorial identities and recursions satisfied by these 
rooted chromatic polynomials, explain their relation to pointed chromatic
functions and rooted $U$-polynomials, and prove three main theorems.
First, for all non-empty connected graphs $G$, Stanley's polynomial
$X_G(x_1,\ldots,x_N)$ is irreducible in $\Q[x_1,\ldots,x_N]$ for all
large enough $N$. The same result holds for our rooted variant where
the root node must avoid a specified color. We prove irreducibility
by a novel combinatorial application of Eisenstein's Criterion. 
Second, we prove the rooted version of Stanley's Conjecture: two rooted trees 
are isomorphic as rooted graphs if and only if their rooted chromatic 
polynomials are equal.  In fact, we prove that a one-variable specialization 
of the rooted chromatic polynomial (obtained by setting $x_0=x_1=q$, 
$x_2=x_3=1$, and $x_n=0$ for $n>3$) already distinguishes rooted trees.
Third, we answer a question of Pawlowski by providing a combinatorial 
interpretation of the monomial expansion of pointed chromatic functions.
\end{abstract}


\maketitle 

\section{Introduction}
\label{sec:intro}

\subsection{Stanley's Chromatic Symmetric Functions}
\label{subsec:stan-CSF}

Let $G$ be a simple graph with vertex set $V(G)$ and edge set $E(G)$.
A \emph{proper coloring} of $G$ is a coloring of the vertices in $V(G)$
such that adjacent vertices receive different colors. 
The classical \emph{chromatic polynomial} $\chi_G$ counts
proper colorings: $\chi_G(k)$ is the number of
proper colorings of $G$ when there are $k$ available colors.
In 1995, Richard Stanley~\cite{stanley} defined the chromatic
symmetric function $X_G$, which is a multivariable generalization 
of $\chi_G$.
For fixed $N\geq 1$, $X_G(x_1,x_2,\ldots,x_N)$ is
a weighted sum of proper colorings of $G$ using colors in $\{1,2,\ldots,N\}$.
A proper coloring that colors $e_i$ vertices with color $i$ (for 
$1\leq i\leq N$) contributes the monomial $x_1^{e_1}x_2^{e_2}\cdots x_N^{e_N}$
to $X_G(x_1,\ldots,x_N)$. Each $X_G(x_1,\ldots,x_N)$ is a symmetric polynomial 
in $N$ variables.  Taking a formal limit as $N$ goes to infinity
produces the chromatic symmetric function $X_G$ using variables $(x_n:n>0)$.
We recover $\chi_G(k)$ from $X_G$ by setting $x_i=1$ for $1\leq i\leq
k$ and $x_i=0$ for all $i>k$.

The chromatic symmetric functions have been a rich source
of combinatorial insights and problems.  The influential
paper~\cite[pg. 170]{stanley} contains an open problem that 
we call Stanley's Chromatic Conjecture for Trees:
\begin{conjecture}\label{conj:stan}
 Trees $T$ and $U$ are isomorphic (as graphs) if and only if $X_T=X_U$.
\end{conjecture}
This conjecture remains unsettled, although progress has been made in
several directions. Using computer investigations,
Heil and Ji~\cite{heil-ji} verified that the conjecture is true
for all trees with at most 29 vertices. 
The conjecture also holds for certain restricted classes of trees: 
spiders~\cite{martin-morin-wagner}, 2-spiders~\cite{huryn-chmutov},
caterpillars~\cite{aliste-prieto-zamora,loebl-sereni,morin}, trivially
perfect graphs~\cite{tsujie}, and trees with diameter at most
five~\cite{alieste}; see also the families described 
in~\cite[Cor. 8.7]{AWW-Laplacian}. It is also known
that features of a tree such as its degree sequence, number of leaves,
and path sequence are recoverable from $X_T$~\cite{martin-morin-wagner}.

In this paper, we study some variations of Stanley's chromatic 
symmetric function for rooted graphs.  Given a rooted graph $G_*$, 
we define polynomials $X_i(G_*)$ and $X_{\neq i}(G_*)$
that are weighted sums of certain proper colorings of $G_*$. 
For $0\leq i\leq N$, $X_i(G_*;x_0,x_1,\ldots,x_N)$
is a sum over proper colorings of $G_*$
using available colors $\{0,1,\ldots,N\}$ such that the root
vertex of $G$ must receive color $i$. 
The polynomial $X_{\neq i}(G_*;x_0,x_1,\ldots,x_N)$ is defined similarly,
but we sum over proper colorings where the root cannot receive color $i$.
We can view $X_0(G_*;x_0,\ldots,x_N)$
as a polynomial in $x_0$ where the coefficient of each $x_0^k$ is a
symmetric polynomial in $x_1,\ldots,x_N$. 

We prove various combinatorial and algebraic properties of these
polynomials, including: a recursion for rooted trees
(Proposition~\ref{prop:recT}), a deletion-contraction recursion for
rooted graphs (Theorem~\ref{thm:dc}), combinatorial interpretations of
the coefficients of $x_0^k$ in $X_0(G_*)$ and $X_{\neq 0}(G_*)$
(Proposition~\ref{prop:coeff-zk}), and a formula for the power-sum
expansion of $X_0(G_*)$ (Proposition~\ref{prop:p-exp}). 
Using the power-sum formula, we show how to obtain $X_0(G_*)$ as an
algebraic transformation of pointed chromatic functions 
or of the rooted $U$-polynomials
(Sections~\ref{subsec:pawlowski-SF} and~\ref{subsec:rooted-U}).
Theorem~\ref{thm:best} proves this rooted version of Stanley's Conjecture:
for all $N\geq 3$, two rooted trees $T_*$ and $U_*$ are isomorphic 
(as rooted graphs) if and only if 
$X_0(T_*;x_0,\ldots,x_N)=X_0(U_*;x_0,\ldots,x_N)$.
In fact, we prove that the one-variable specialization 
$f_0(G_*)=X_0(G_*;q,q,1,1)$ already suffices to distinguish rooted trees.

\subsection{Irreducibility Properties}
\label{subec:irred}

Irreducibility and unique factorization are crucial algebraic 
tools in our proof that $X_0(G_*)$ and its specialization $f_0(G_*)$
distinguish rooted trees.
This proof strategy is also used in prior literature such 
as~\cite[Theorem 9]{ADZ-rooted},~\cite[Theorem 2.8]{gordon-mcmahon}, 
and~\cite[Theorem 1.3]{hasebe-tsujie}.  
Suppose a graph $G$ has connected components $C_1,C_2,\ldots,C_k$.
Then $X_G=X_{C_1}X_{C_2}\cdots X_{C_k}$, since each proper coloring of
$G$ arises by independently choosing a proper coloring for each component.
This means that the chromatic symmetric polynomial for a disconnected
graph is always reducible.  A natural question is whether $X_G$ must be 
irreducible when $G$ is a connected graph or a tree. 

To formulate precise questions about irreducibility, we must specify the ring 
in which the factorization occurs. The polynomial $X_G(x_1,\ldots,x_N)$
belongs to the polynomial ring $\Q[x_1,\ldots,x_N]$ and also to the
subring $\Lambda_N$ of symmetric polynomials in $N$ variables. 
Both rings are unique factorization domains (UFDs); in fact, 
we can view $\Lambda_N$ as a polynomial ring $\Q[e_1,\ldots,e_N]$ in the
algebraically independent elementary symmetric polynomials $e_1,\ldots,e_N$.

Using work of Cho and van Willigenburg~\cite{cho-vanWilligenburg},
Tsujie~\cite[Cor.~2.4]{tsujie} proved that for any non-empty connected 
graph $G$ and all $N\geq |V(G)|$,
$X_G(x_1,\ldots,x_N)$ is irreducible in the ring $\Lambda_N$.  In
Theorem~\ref{thm:XG-irred}, we prove that for any non-empty connected
graph $G$, $X_G(x_1,\ldots,x_N)$ is irreducible in $\Q[x_1,\ldots,x_N]$
for large enough $N$. Note that irreducibility in $\Q[x_1,\ldots,x_N]$
is a stronger condition than irreducibility in $\Lambda_N$, since
there are more potential irreducible factors in the full polynomial
ring compared to its subring $\Lambda_N$. For example, $e_N=x_1x_2\cdots x_N$
is irreducible in $\Lambda_N$ but reducible in $\Q[x_1,\ldots,x_N]$.

The core of our proof
(Lemma~\ref{lem:use-eisen1}) invokes Eisenstein's Criterion to show
the specialization $X_G(q,\ldots,q,1,\ldots,1)$ is irreducible in $\Q[q]$ 
if we set $k$ variables equal to $q$ and $p$ variables equal to $1$
for certain choices of $k$ and $p$.  The proof succeeds because of a 
surprising affinity between the hypotheses of Eisenstein's Criterion and 
the enumerative properties of proper colorings.
Theorem~\ref{thm:better} uses the same ideas to prove that
for connected bipartite rooted graphs $G_*$,
the specialization $X_{\neq 0}(G_*;q,q,1,1,\ldots,1)$ 
is irreducible in $\Q[q]$ if we set $p$ variables equal to $1$ 
for certain primes $p$.
When $G_*$ is a rooted tree, irreducibility holds for all primes $p$.

\subsection{Variations of Chromatic Symmetric Polynomials}
\label{subsec:related-polys}

Stanley's chromatic symmetric function is just one
member of an entire ecosystem of polynomials and algebraic constructs
based on graph colorings. Some important variations of $X_G$
include the noncommutative chromatic symmetric functions of Gebhard
and Sagan~\cite{gebhard-sagan}, the chromatic quasisymmetric functions
of Shareshian and Wachs~\cite{shareshian-wachs}, the strict order
quasisymmetric functions of Hasebe and Tsujie~\cite{hasebe-tsujie},
the pointed chromatic functions of Pawlowski~\cite{pawlowski},
and the rooted $U$-polynomials of Aliste-Prieto, de Mier, and 
Zamora~\cite{ADZ-rooted}, among many others (see Section~\ref{sec:others}
for a fuller discussion). Our polynomials $X_0(G_*)$ are related
to the pointed chromatic functions by an algebraic transformation
on $\Lambda_N$, as explained in Proposition~\ref{prop:X0-vs-paw}.
In turn, the pointed chromatic functions are specializations of
the rooted $U$-polynomials, which are known to distinguish
rooted trees~\cite[Theorem 9]{ADZ-rooted}. This leads to a second,
algebraic proof that $X_0(G_*;x_0,\ldots,x_N)$ distinguishes
rooted trees for $N\geq |V(G)|$.

Previous authors have found other multivariate polynomial invariants
characterizing isomorphism classes of rooted trees.
In addition to the rooted $U$-polynomials already mentioned,
each of the following is a complete invariant for rooted trees:
the two-variable Chaudhary--Gordon polynomials~\cite{chaudhary-gordon},
the polychromates of Bollob\'as and Riordan~\cite{bollabas-riordan}, 
the greedoid polynomials of Gordon and McMahon~\cite{gordon-mcmahon}, 
and the strict order quasisymmetric functions of Hasebe and 
Tsujie~\cite{hasebe-tsujie}. Related results appear 
in~\cite{ellzey,foley-et-al,gebhard-sagan,shareshian-wachs}.

\begin{remark}\label{rem:heil-ji}
The polynomials $Z_G^v(c)$ defined by Heil and
Ji~\cite[Def. 3.1]{heil-ji} are essentially the same as our
polynomials $X_c(G_*^v)$.  However, given Heil and Ji's algorithmic
focus, there is no overlap with this paper beyond the decomposition
of~\eqref{eq:recT}.  We note for completeness that in~\cite{heil-ji},
Lemma~3 and the following corollaries must be modified slightly to
account for the fact that the $F_i$ are symmetric functions in all the
variables \emph{except} $x_c$.
\end{remark}

Given the many prior variants of Stanley's chromatic symmetric
functions, we should highlight some particular benefits of the
polynomials $X_0(G_*)$ studied here. First, $X_0(G_*)$ is
combinatorially very close to the original polynomial $X_G$ --- the
only new restriction on proper colorings is that the root vertex must
get color $0$.  Second, we can easily recover $X_G$ from $X_0(G_*)$ by
the symmetry operations specified in~\eqref{eq:reln}
and~\eqref{eq:reln2}.  Proposition~\ref{prop:recover} shows that being
able to solve the reverse problem (recovering all $X_0(G_*)$ that
arise from $X_G$ by varying the choice of root) is equivalent to
Stanley's Conjecture~\ref{conj:stan}.  Third, the deletion-contraction
recursion for $X_0(G_*)$, which is not directly available for $X_G$,
lets us efficiently compute both polynomials. Fourth, the simple
restriction on the color of the root vertex suffices to yield a short
proof of the desired invariance property for rooted trees.  The proof
makes compelling contact between a combinatorial property (isomorphism
of rooted trees) and an algebraic property (unique factorization of
polynomials into irreducible factors).  Fifth, the specialized version
$X_0(G_*;q,q,1,1)$ discards a huge amount of information on colorings
but still suffices to distinguish rooted trees.

Some previously studied relatives of $X_G$ achieve similar invariance results 
by incorporating additional information in other ways.  
The Gebhard--Sagan noncommutative chromatic symmetric 
function~\cite{gebhard-sagan} remembers which vertices 
receive which colors by using non-commuting variables.
The Shareshian--Wachs chromatic quasisymmetric 
polynomial~\cite{shareshian-wachs}
uses a new variable $t$ to record an ascent statistic for each proper coloring.
Hasebe and Tsujie~\cite{hasebe-tsujie} adopt a related approach where
only proper colorings satisfying additional ascent constraints are allowed.
Section~\ref{sec:others} contains a more detailed discussion.

\subsection{Outline of Paper}
\label{subsec:outline}

The outline of this paper is as follows. Section~\ref{sec:defs}
introduces needed definitions, notation, and background.
Section~\ref{sec:proof} proves our irreducibility theorems
for $X_G$, $X_{\neq 0}(G_*)$, and their specializations.
These results are used to prove our analogue of Stanley's Conjecture
(Theorem~\ref{thm:best}).
Section~\ref{sec:additional} proves the deletion-contraction recursion
for $X_0(G_*)$ and describes the coefficients of $X_0(G_*)$ in various
bases.  Section~\ref{sec:others} examines some previously studied
variants of Stanley's chromatic symmetric function and their relations
to the rooted chromatic polynomials.

\section{Definitions and Background}
\label{sec:defs}

\subsection{Definition of Polynomials for Rooted Graphs}
\label{subsec:def-rooted}

The notation $G=(V(G),E(G))$ means $G$ is a simple graph with vertex
set $V(G)$ and edge set $E(G)$.  A \emph{rooted graph} $G_*$ is a
nonempty graph $G$ with one vertex $r$ of $G$ marked as the
\emph{root}. When we need to display the root, 
we write $G_{*}^r$ for the rooted graph obtained
from $G$ with root vertex $r$.  The \emph{color set} is
$C=\{0,1,2,\ldots,N\}$, where $N$ is a fixed positive integer.
 A \emph{proper coloring} of $G$ is a function
$\kappa:V(G)\rightarrow C$ such that for all $v,w\in V(G)$, if an edge
joins $v$ to $w$, then $\kappa(v)\neq \kappa(w)$.  Let $\Col(G)$ be
the set of proper colorings of $G$.  For each color $i\in C$, let
$\Col_i(G_*)$ be the set of proper colorings of $G$ where
$\kappa(r)=i$ (the root gets color $i$).  Let $\Col_{\neq i}(G_*)$ be
the set of proper colorings of $G$ where $\kappa(r)\neq i$ (the root's
color is not $i$).  The \emph{weight} of a coloring
$\kappa:V(G)\rightarrow C$ is $\wt(\kappa)=\prod_{v\in V(G)}
x_{\kappa(v)}$. 

We now introduce several versions of chromatic polynomials for a rooted
graph $G_*$. These are polynomials in $\Q[x_0,x_1,\ldots,x_N]$
with nonnegative integer coefficients. Define:
\begin{align}
 X(G_*;x_0,x_1,\ldots,x_N) &=\sum_{\kappa\in\Col(G)} \wt(\kappa);\\
 X_i(G_*;x_0,x_1,\ldots,x_N) &= \sum_{\kappa\in\Col_i(G_*)} \wt(\kappa); 
\label{eq:def-orig}\\
 X_{\neq i}(G_*;x_0,x_1,\ldots,x_N)
 &= \sum_{\kappa\in\Col_{\neq i}(G_*)} \wt(\kappa).
\end{align}
We omit the variable list from the notation when it is clear from context.
Note $X(G_*;x_0,x_1,\ldots,x_N)$ is Stanley's chromatic 
symmetric function $X_G$ specialized to the given variable set. 
This polynomial is symmetric in $x_0,\ldots,x_N$,
since applying any permutation of the color set $C$ to a proper coloring
of $G$ produces another proper coloring of $G$.
For each $i\in C$, $X_i(G_*)$ and $X_{\neq i}(G_*)$ are polynomials in
$x_0,\ldots,x_N$ that are symmetric in all the variables except $x_i$.
This follows since any permutation of $\{0,1,\ldots,N\}$ fixing $i$
induces bijections from $\Col_i(G_*)$ to itself and from $\Col_{\neq i}(G_*)$ 
to itself.  We refer to $X_0(G_*)$ as the \emph{rooted chromatic
polynomial} for $G_*$ in $N+1$ variables.  

Let $\Lambda_N$ be the ring of symmetric
polynomials in variables $x_1,\ldots,x_N$, 
let $\Lambda$ be the ring of symmetric functions in $(x_k:k\geq 1)$,
and let $z=x_0$. 
Then $X_0(G_*)$ and $X_{\neq 0}(G_*)$ are in $\Lambda_N[z]$, 
the ring of polynomials in $z$ with coefficients in $\Lambda_N$.  
The constant coefficient of $X_{\neq 0}(G_*)$, namely the specialization 
upon setting $z=0$, is Stanley's chromatic symmetric polynomial $X_G$ in
variables $x_1,\ldots,x_N$. Thus, we may view $X_{\neq 0}(G_*)$
as a refinement of the original chromatic symmetric function.
The coefficient of $z^k$ in $X_0(G_*)$ 
is a symmetric polynomial in $x_1,\ldots,x_N$ 
that is homogeneous of degree $n-k$, where $n=|V(G_*)|$.
For all $N$ and $M$ with $N>M$, setting
the last $N-M$ variables equal to $0$ in 
$X_0(G_*;x_0,x_1,\ldots,x_N)$ produces $X_0(G_*;x_0,x_1,\ldots,x_M)$.
Because of this stability property, we can let the number of variables
tend to infinity to obtain a version of $X_0(G_*)$ in $\Lambda[z]$.
Informally, this polynomial in $z$ (with symmetric function coefficients)
contains exactly the same information as each finite version
$X_0(G_*;x_0,x_1,\ldots,x_N)$ where $N\geq |V(G_*)|$.
Similar comments are true for $X_{\neq 0}(G_*)$.

\begin{example}\label{ex:path3a}
Suppose $N=2$ and $G$ is a three-vertex path. Then
\begin{equation}\label{eq:ex-xt}
  X_{\includegraphics[width=1cm]{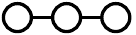}}(x_0,x_1,x_2)
 = 6x_0x_1x_2 + x_0^2x_1 + x_0^2x_2 + x_0x_1^2 + x_1^2x_2 + x_0x_2^2+x_1x_2^2.
\end{equation}
For instance, the coefficient of $x_0x_2^2$ is $1$ because there is only $1$
proper coloring of $G$ using color $2$ twice and color $0$ once
(the middle vertex must receive color $0$).  
Now consider rooted graphs with underlying graph $G$.  
Choosing the root to be either endpoint of the path gives
\begin{align}
  X_0(\raisebox{-0.1\height}{\includegraphics[width=1cm]{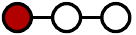}}) = X_0(\raisebox{-0.1\height}{\includegraphics[width=1cm]{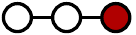}}) 
&= 2x_0x_1x_2 + x_0^2x_1 + x_0^2x_2;\label{eq:exA}\\
  X_{\neq 0}(\raisebox{-0.1\height}{\includegraphics[width=1cm]{ex-A}}) 
 = X_{\neq 0}(\raisebox{-0.1\height}{\includegraphics[width=1cm]{ex-C}}) 
&= 4x_0x_1x_2+x_0x_1^2+x_0x_2^2+x_1^2x_2+x_1x_2^2.
\end{align}
On the other hand, choosing the root to be the middle vertex yields
\begin{equation}\label{eq:exB}
 X_0(\raisebox{-0.1\height}{\includegraphics[width=1cm]{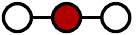}}) = 
x_0(2x_1x_2+x_1^2 + x_2^2);\quad
  X_{\neq 0}(\raisebox{-0.1\height}{\includegraphics[width=1cm]{ex-B}}) = 
 4x_0x_1x_2+x_0^2x_1+x_0^2x_2+x_1x_2^2+x_1^2x_2. 
\end{equation}
\end{example}

\subsection{Basic Identities}
\label{sec:identities}

We now establish some identities relating the various chromatic polynomials
just defined.
Let $\Sym$ be the symmetric group on the color set $C=\{0,1,\ldots,N\}$.
For $i\neq j$ in $C$, $(i,j)$ is the transposition in $\Sym$
that interchanges $i$ and $j$.
The group $\Sym$ acts on $\Q[x_0,x_1,\ldots,x_N]$ by permuting the
variables: $\sigma\bullet x_i=x_{\sigma(i)}$ for $\sigma\in\Sym$ and $i\in C$.
The following identities follow directly from the definitions and
symmetry arguments. 

\begin{prop} For all rooted graphs $G_*$ and all $k\in\{0,1,\ldots,N\}$,
\begin{align}
 X(G_*) &= \sum_{i=0}^N X_i(G_*)=X_k(G_*)+X_{\neq k}(G_*);\label{eq:reln} \\
 X_k(G_*) &= (0,k)\bullet X_0(G_*); \label{eq:reln2} \\
 X_{\neq 0}(G_*) &= \sum_{i=1}^N X_i(G_*)=\sum_{i=1}^N (0,i)\bullet X_0(G_*);\\
 X_{\neq k}(G_*) &= (0,k)\bullet X_{\neq 0}(G_*).
\end{align}
Hence, $X(G_*)=X_G$, $X_{\neq 0}(G_*)$, $X_k(G_*)$, 
and $X_{\neq k}(G_*)$ can all be recovered 
from the rooted chromatic polynomial $X_0(G_*)$. 
\end{prop}

A more interesting fact is that knowledge of $X_{\neq 0}(G_*)$ is also
sufficient to recover $X_0(G_*)$ and all other polynomials listed here.
Given $X_{\neq 0}(G_*)$, we first obtain 
$X_{\neq i}(G_*)=(0,i)\bullet X_{\neq 0}(G_*)$ for all $i$
between $1$ and $N$.  The key observation is:
\begin{equation*}
\sum_{i=0}^N X_{\neq i}(G_*)
  =\sum_{i=0}^N\sum_{\substack{j=0\\j\neq i}}^N X_j(G_*)
  =\sum_{j=0}^N X_j(G_*)\sum_{\substack{i=0\\i\neq j}}^N 1
  =N\sum_{j=0}^N X_j(G_*)=NX(G_*).
\end{equation*}
So $X(G_*)=\frac{1}{N}\sum_{i=0}^N (0,i)\bullet X_{\neq 0}(G_*)$. 
Combining this with~\eqref{eq:reln} yields the following.

\begin{prop}\label{prop:x0-det} For any rooted graph $G_*$,
\begin{equation}\label{eq:x0-det}
  X_0(G_*)=\frac{\sum_{i=0}^N (0,i)\bullet X_{\neq 0}(G_*)}{N}-X_{\neq 0}(G_*).
\end{equation}
\end{prop}

\subsection{Isomorphisms of Graphs and Rooted Graphs}
\label{subsec:iso-defs}

Given graphs $G=(V(G),E(G))$ and $H=(V(H),E(H))$,
a \emph{graph isomorphism} from $G$ to $H$ is a bijection
$f:V(G)\rightarrow V(H)$ such that for all $v,w\in V(G)$,
$E(G)$ contains the edge joining $v$ and $w$ if and only if
$E(H)$ contains the edge joining $f(v)$ and $f(w)$.
Suppose $G_*$ is a rooted graph with underlying graph $G$ and root $r$,
and $H_*$ is a rooted graph with underlying graph $H$ and root $s$.
A \emph{rooted graph isomorphism} from $G_*$ to $H_*$
is a graph isomorphism $f:V(G)\rightarrow V(H)$ with $f(r)=s$.
When such an isomorphism exists, we write $G_*\cong H_*$.

The concept of rooted graph isomorphism for rooted trees
is closely related to the following recursive
construction of rooted trees. For any rooted tree $T_*$, either
$T_*$ is a one-vertex graph consisting of the root $r$ and no edges; 
or $T_*$ has root $r$ joined by edges to the roots of $c\geq 1$
\emph{principal subtrees} $T_{1*},T_{2*},\ldots,T_{c*}$. 
Note that these subtrees can be listed in any order
(in contrast to subtrees in an ordered plane tree, where
 the order of the children of the root is significant).
Studying isomorphism classes of rooted trees amounts to erasing
all vertex labels. Upon doing this, the principal subtrees of $T_*$
form a \emph{multiset} of isomorphism classes of rooted trees, 
in which order does not matter, but repetitions are allowed.

It is routine to check the following criterion for two rooted
trees $T_*$ and $U_*$ (with more than one vertex)
to be isomorphic as rooted graphs.
If $T_*$ has principal subtrees $T_{1*},\ldots,T_{c*}$
and $U_*$ has principal subtrees $U_{1*},\ldots,U_{d*}$, then:
\begin{equation}
\label{eq:iso-crit}
\text{$T_*\cong U_*$ if and only if $c=d$ and (after suitable 
reordering) $T_{i*}\cong U_{i*}$ for $1\leq i\leq c$.}
\end{equation}

\subsection{Recursion for $X_0(T_*)$}
\label{subsec:rec-rooted-tree}

There is a simple recursion for computing $X_0(T_*)$ when $T_*$ is a
rooted tree. A version of this result was used for computational
purposes by Heil and Ji~\cite{heil-ji}
(cf.~\cite{bollabas-riordan,chaudhary-gordon}).

\begin{prop}\label{prop:recT}
Let $T_*$ be a rooted tree with principal rooted subtrees
$T_{1*},T_{2*},\ldots,T_{c*}$. Then
\begin{equation}\label{eq:recT}
 X_0(T_*)=x_0\prod_{j=1}^c X_{\neq 0}(T_{j*}). 
\end{equation}
\end{prop}
\begin{proof}
The left side $X_0(T_*)$ is the weighted sum of all colorings
in $\Col_0(T_*)$. We build each such coloring uniquely by making
the following choices. First, color the root vertex $r$ with color $0$;
the weight monomial for this step is $x_0$.
Next, color the subtree $T_{1*}$ using any coloring in 
$\Col_{\neq 0}(T_{1*})$; color the subtree $T_{2*}$ using any coloring in
$\Col_{\neq 0}(T_{2*})$; and so on. The colorings of the subtrees
can be chosen independently since the subtrees do not connect with
each other except through the root $r$. The recursion~\eqref{eq:recT}
follows by the Product Rule for Weighted Sets~\cite[\S 5.8]{loehr-comb}.
\end{proof}

\section{Irreducibility Results}
\label{sec:proof}

The main result proved in this section is the following theorem.

\begin{theorem}\label{thm:XG-irred}
For any non-empty connected graph $G$, there is an integer $M$
such that for all $N\geq M$:
\\ (a) $X_G(x_1,\ldots,x_N)$ is irreducible in $\Q[x_1,\ldots,x_N]$.
\\ (b) $X_{\neq 0}(G_*;x_0,\ldots,x_N)$ is irreducible in $\Q[x_0,\ldots,x_N]$.
\end{theorem}

The proof relies on the following properties of the one-variable
chromatic function $\chi_G$, where $\chi_G(k)$ is the number of
proper colorings of $G$ using $k$ available colors.
First, for any non-empty graph $G$,
$\chi_G$ is a polynomial in $\Z[x]$ that is divisible by $x$,
and $\chi_G(k)>0$ for some positive integer $k$. 
The least such $k$ is the \emph{chromatic number} of $G$.
Second, for any non-empty connected graph $G$, the coefficient $c_1$
of $x$ in $\chi_G$ is nonzero~\cite[Exc. 5.3.10, pg. 230]{west-graph}.
The proof of Theorem~\ref{thm:XG-irred}
will show that we may take $M$ in that theorem
to be $k+p$, where $k$ is the chromatic number of $G$ and
$p$ is the least prime such that $p$ does not divide $\chi_G(k)$
and $p$ does not divide $c_1$.  If $G$ is an $n$-vertex tree with $n>1$, 
then $\chi_G=x(x-1)^{n-1}=x^n+\cdots+(-1)^{n-1}x$, and $G$ has
chromatic number $2$. So when $G$ is a tree, we may take
$k=2$ and $p=3$ to see that Theorem~\ref{thm:XG-irred} holds with $M=5$.

\subsection{Background on Irreducibility and Specializations}
\label{sec:spec-irred}

This subsection collects some basic facts about irreducible polynomials
needed for our proof of Theorem~\ref{thm:XG-irred}. The key tool 
used is Eisenstein's Criterion for irreducibility of polynomials in $\Q[q]$.

\begin{theorem} 
(Eisenstein's Criterion~\cite[Thm. 6.15]{hungerford})\label{thm:eisen}
 Suppose $f=a_0+a_1q+a_2q^2+\cdots+a_nq^n \in\Q[q]$ 
 is a polynomial of degree $n>0$ with integer coefficients,
 $p$ is a prime, $p$ divides $a_k$ for $0\leq k<n$,
 $p$ does not divide $a_n$, and $p^2$ does not divide $a_0$.
 Then $f$ is irreducible in the polynomial ring $\Q[q]$.
\end{theorem}

\begin{lemma}\label{lem:homog}
Suppose $f\in\Q[x_0,\ldots,x_N]$ is reducible and homogeneous of degree
$n>0$. Then $f$ has a proper factorization $f=gh$ 
for some non-constant homogeneous polynomials $g,h\in\Q[x_0,\ldots,x_N]$.
\end{lemma}
\begin{proof}
By reducibility of $f$, there is a proper factorization
$f=gh$ where $g,h\in\Q[x_0,\ldots,x_N]$ are not initially known
to be homogeneous.
Write $g=g^{(d)}+\text{(lower terms)}$, where $d>0$ is the maximum
degree of any monomial in $g$, and $g^{(d)}$ is the sum of all
monomials in $g$ having degree $d$. Write
$h=h^{(e)}+\text{(lower terms)}$, where $e>0$ is the maximum degree of
any monomial in $h$, and $h^{(e)}$ is the sum of all monomials in $h$
having degree $e$.  Then $gh=g^{(d)}h^{(e)}+\text{(lower
  terms)}$. Since $\Q[x_0,\ldots,x_N]$ is an integral domain, the
product $g^{(d)}h^{(e)}$ (which is homogeneous of degree $d+e$) cannot
be zero. Since $f$ is homogeneous, this product must equal $f$. 
So $f=g^{(d)}h^{(e)}$, where both factors are homogeneous and non-constant.
\end{proof}

Write $f|_{x_i=0}$ for the specialization of a polynomial
$f$ obtained by setting variable $x_i$ equal to $0$.

\begin{lemma}\label{lem:xi-to-0}
Suppose $f\in\Q[x_0,x_1,\ldots,x_N]$ is homogeneous of degree $n>0$. 
If $f|_{x_i=0}$ is an irreducible polynomial
 in $\Q[x_0,x_1,\ldots,\hat{x}_i,\ldots x_N]$
(where $\hat{x}_i$ means variable $x_i$ is omitted), 
then $f$ is irreducible in $\Q[x_0,x_1,\ldots,x_N]$.
\end{lemma}
\begin{proof}
Let $f'=f|_{x_i=0}$.  To get a contradiction, assume $f'$ is
irreducible (hence non-constant and nonzero) and $f$ is reducible.
By the previous lemma, there is a proper factorization
$f=gh$ where $g$ and $h$ are both homogeneous of positive degree.
Letting $g'=g|_{x_i=0}$ and $h'=h|_{x_i=0}$, we have $f'=g'h'$.
Since $f'\neq 0$, $g'$ and $h'$ are nonzero. It follows at once that
$g'$ and $h'$ are homogeneous of positive degree. Thus $f'=g'h'$
is a proper factorization of $f'$, which is a contradiction.
\end{proof}

Homogeneity of $f$ is needed in Lemma~\ref{lem:xi-to-0}.
For example, $f=x(y+1)$ is reducible in $\Q[x,y]$,
yet $f|_{y=0}=x$ is irreducible in $\Q[x]$.

\begin{lemma}\label{lem:irred-rels}
Let $G_*$ be a nonempty rooted graph with chromatic number $k$.  
For all $N\geq M\geq k$:
\begin{enumerate}
\item If $X_G(x_1,\ldots,x_M)$ is irreducible, then 
 $X_{\neq 0}(G_*;x_0,x_1,\ldots,x_M)$ is irreducible.
\item If $X_G(x_1,\ldots,x_M)$ is irreducible, then
  $X_G(x_1,\ldots,x_N)$ is irreducible.
\item If $X_{\neq 0}(G_*;x_0,\ldots,x_M)$ is irreducible, then
  $X_{\neq 0}(G_*;x_0,\ldots,x_N)$ is irreducible.
\end{enumerate}
\end{lemma}
\begin{proof}
All polynomials mentioned in the lemma are homogeneous of degree
$|V(G_*)|$ and are nonzero, since the hypothesis $N\geq M\geq k$
ensures there are enough colors to produce at least one
coloring satisfying all needed conditions. 
Part~(1) follows from Lemma~\ref{lem:xi-to-0}, since setting $x_0=0$ in 
$X_{\neq 0}(G_*;x_0,\ldots,x_M)$ produces $X_G(x_1,\ldots,x_M)$.
Parts~(2) and (3) follow from Lemma~\ref{lem:xi-to-0}
by setting $x_N=0,\ldots,x_{M+1}=0$ one at a time.
\end{proof}

\subsection{Proof of Theorem~\ref{thm:XG-irred}}
\label{subsec:prove-XG-irred}

This section proves Theorem~\ref{thm:XG-irred}(a).
Theorem~\ref{thm:XG-irred}(b) follows at once from
part~(a) and Lemma~\ref{lem:irred-rels}. 
We begin with a lemma containing the core idea of the proof.

\begin{lemma}\label{lem:use-eisen1}
Let $G$ be a non-empty connected graph with chromatic polynomial
$\chi_G\in\Z[x]$ and chromatic number $k$.
Let $p$ be a prime such that $p$ does not divide $\chi_G(k)$ and
$p$ does not divide the coefficient of $x$ in $\chi_G$. Let $M=k+p$.
Let $f\in\Q[q]$ be the specialization of $X_G(x_1,\ldots,x_M)$
obtained by setting $x_1=x_2=\cdots=x_k=q$ 
and $x_{k+1}=x_{k+1}=\cdots=x_{k+p}=1$.
Then $f$ is irreducible in $\Q[q]$.
\end{lemma}
\begin{proof}
Let $n=|V(G)|$; we may assume $n>1$. Note that $f=\sum_{j=0}^n a_jq^j$,
where $a_j\in\Z$ is the number of proper colorings 
$\kappa:V(G)\rightarrow\{1,2,\ldots,M\}$ such that 
the number of $v$ with $\kappa(v)\in\{1,2,\ldots,k\}$ is $j$. 
This holds since specializing $X_G$ to $f$
replaces the weight monomial $x_1^{e_1}x_2^{e_2}\cdots x_M^{e_M}$ of 
each proper coloring $\kappa$ by $q^{e_1+e_2+\cdots+e_k}$, 
where $e_1+\cdots+e_k$ is the number of vertices that 
$\kappa$ colors using $1,2,\ldots,k$.

The following counting argument leads to a formula for $a_j$,
where $0\leq j\leq n$ is fixed.
To build a proper coloring counted by $a_j$,
first choose a $j$-element subset $S$ of $V(G)$,
and let $S^c=\{v\in V(G): v\not\in S\}$.
Let $G|S$ be the graph with vertex set $S$ and edge set
consisting of all edges in $G$ joining two vertices in $S$,
and define $G|S^c$ similarly.
Choose a proper coloring $\kappa_1$ of $G|S$ using color set $\{1,2,\ldots,k\}$,
and choose a proper coloring $\kappa_2$ of $G|S^c$ 
using color set $\{k+1,k+2,\ldots,k+p\}$. The union of $\kappa_1$ and $\kappa_2$
is a proper coloring of $G$ counted by $a_j$, and all such colorings
arise uniquely by this construction process. In summary,
\begin{equation}\label{eq:aj-formula}
 a_j=\sum_{\substack{S\subseteq V(G):\\ |S|=j}} 
\chi_{G|S}(k)\chi_{G|S^c}(p). 
\end{equation}
In particular, $a_n=\chi_G(k)$ and $a_0=\chi_G(p)$.

We now confirm that $f$ satisfies the hypotheses of Eisenstein's Criterion
(Theorem~\ref{thm:eisen}). By choice of $k$ and $p$, $a_n$ is nonzero
and not divisible by $p$.  We know $\chi_G$ is $c_1x$ plus higher terms all 
divisible by $x^2$, and $p$ does not divide $c_1$. Replacing $x$ by $p$, 
we see that $\chi_G(p)$ is $c_1p$ plus a multiple of $p^2$, so $p^2$ does not
divide $a_0=\chi_G(p)$. For $0\leq j<n$, each subgraph $S^c$ used in the sum
for $a_j$ is nonempty with $n-j>0$ vertices. 
So each $\chi_{G|S^c}$ is divisible by $x$,
   each $\chi_{G|S^c}(p)$ is divisible by $p$, and $a_j$ itself is
therefore divisible by $p$. By Eisenstein's Criterion,
$f$ is irreducible in $\Q[q]$.
\end{proof}

Lemma~\ref{lem:use-eisen1} suggests that $X_G(x_1,\ldots,x_M)$ is 
irreducible, since its specialization $f(q)$ is irreducible.
However, we need an additional technical argument to rule out 
the possibility that a proper factorization of $X_G$ happens to
specialize to a trivial factorization of $f(q)$. As an
example of this phenomenon, note $(x_1+x_2)(x_3+x_4+x_5)$ is 
reducible in $\Q[x_1,\ldots,x_5]$, but specializing $x_1=x_2=q$
and $x_3=x_4=x_5=1$ yields the irreducible polynomial $6q$ in $\Q[q]$.
As another example, the symmetric polynomial 
$(-x_1+x_2+x_3)(x_1-x_2+x_3)(x_1+x_2-x_3)$ is reducible in $\Q[x_1,x_2,x_3]$,
but setting $x_1=x_2=q$ and $x_3=1$ yields the irreducible polynomial
$2q-1$ in $\Q[q]$.  The next lemma addresses this issue and, combined with
Lemma~\ref{lem:use-eisen1} and Lemma~\ref{lem:irred-rels}(2), 
completes the proof of Theorem~\ref{thm:XG-irred}(a).

\begin{lemma}\label{lem:new-specz}
Suppose $2\leq k\leq M$ and $F(x_1,\ldots,x_M)$ is a symmetric polynomial 
in $\Q[x_1,\ldots,x_M]$ with nonnegative coefficients.
Let $q$ be a formal variable and let $f=F(q,\ldots,q,1\ldots,1)\in\Q[q]$ 
be obtained from $F$ by setting $x_1=\cdots=x_k=q$ and $x_{k+1}=\cdots=x_M=1$. 
If $f$ is irreducible in $\Q[q]$, 
then $F$ is irreducible in $\Q[x_1,\ldots,x_M]$.
\end{lemma}
\begin{proof}
We prove the contrapositive. Given $F$ and $f$ as in the lemma setup, assume
there is a proper factorization $F=GH$, where $G,H\in\Q[x_1,\ldots,x_M]$
are not constant. Choose $i,j$ so that $x_i$ appears in $G$ and 
$x_j$ appears in $H$. Since $F$ is symmetric, we can assume
$i=1$ and $j\in\{1,2\}$ by applying a permutation of the variables
to the factorization $F=GH$.

Let $R=\Q[x_{k+1},\ldots,x_M]$ and view $F,G,H$ as elements of the
ring $R[x_1,\ldots,x_k]$. This ring is graded by total degree in
the variables $x_1,\ldots,x_k$. So we can write
\[ F=F^{(d)}+\mbox{(lower terms)},\quad
   G=G^{(d_1)}+\mbox{(lower terms)},\quad
   H=H^{(d_2)}+\mbox{(lower terms)}, \]
where $F^{(d)}$ is the sum of all the highest degree monomials in $F$
(say of degree $d$), and similarly for $G^{(d_1)}$ and $H^{(d_2)}$. 
Since $F=GH$ and there are no zero divisors, we must have
$F^{(d)}=G^{(d_1)}H^{(d_2)}$ and $d=d_1+d_2$. Moreover,
$d_1>0$ since $x_1$ appears in $G$, and $d_2>0$ since $x_1$ or $x_2$
appears in $H$ and $k\geq 2$.

We now apply the specialization that replaces each variable $x_1,\ldots,x_k$
by the new formal variable $q$.  For each $P\in R[x_1,\ldots,x_k]$,
let $P'\in R[q]$ be the image of $P$ under this specialization.
We have $F'=G'H'$ where $F'$, $G'$, $H'$ are in $R[q]$. Write
\[ F'=uq^e+\mbox{(lower terms)},\quad
   G'=rq^{e_1}+\mbox{(lower terms)},\quad
   H'=sq^{e_2}+\mbox{(lower terms)}, \]
where $u,r,s\in R$ are nonzero. Since $R[q]$ is graded by degree in $q$
and has no zero divisors, $u=rs$ and $e=e_1+e_2$. On one hand, $e_1\leq d_1$ 
and $e_2\leq d_2$ since the degree in $q$ of each monomial after the 
specialization matches the original total degree in $x_1,\ldots,x_k$, but 
some terms might perhaps cancel. On the other hand, $e=d$ since all 
coefficients of $F$ are nonnegative. But now $e=e_1+e_2\leq d_1+d_2=d=e$, so
we must actually have $e_1=d_1>0$ and $e_2=d_2>0$. In other words,
$G'$ and $H'$ are non-constant polynomials in $q$. 

For all $p\in R=\Q[x_{k+1},\ldots,x_M]$, 
write $\ov{p}$ for the specialization of $p$ where $x_{k+1}=\cdots=x_M=1$.
Now $u$ has all nonnegative coefficients (since $F$ does),
so $\ov{u}>0$. Then $\ov{r}\cdot\ov{s}=\ov{rs}=\ov{u}>0$,
and hence $\ov{r}\neq 0\neq \ov{s}$.
So when we further specialize the factorization $F'=G'H'$ by
setting $x_{k+1}=\cdots=x_M=1$, we get a factorization
$f=gh$ in $\Q[q]$ where $g$ still has degree $d_1>0$ in $q$
and $h$ still has degree $d_2>0$ in $q$. Thus, $f$ is reducible in $\Q[q]$.
\end{proof}

\subsection{Irreducible One-Variable Specializations of $X_{\neq 0}(G_*)$}
\label{subsec:more-irred-thms}

By modifying the proof idea in Lemma~\ref{lem:use-eisen1}, we can
prove some additional irreducibility results.
For each prime $p$ and rooted graph $G_*$, let $X_i^{2,p}(G_*)\in\Z[q]$
be the specialization of $X_i(G_*;x_0,x_1,\ldots,x_{p+1})$ obtained
by setting $x_0=x_1=q$ and $x_2=\cdots=x_{p+1}=1$. 
Let $X_{\neq i}^{2,p}(G_*)$ be the analogous specialization of 
$X_{\neq i}(G_*;x_0,x_1,\ldots,x_{p+1})$.

\begin{theorem}\label{thm:better}
Let $G_*$ be a connected bipartite rooted graph with $n$ vertices.
For any prime $p$ that does not divide the coefficient of $x$ in $\chi_G$, 
$X_{\neq 0}^{2,p}(G_*)$ is monic of degree $n$ and irreducible in $\Q[q]$.
The conclusion always holds for $p=2$.
When $G_*$ is a rooted tree, the conclusion holds for all primes $p$.
\end{theorem}
\begin{proof}
Write $X_{\neq 0}^{2,p}(G_*)=\sum_{j\geq 0} a_jq^j$.
Adapting the argument leading to~\eqref{eq:aj-formula}, we can show
\begin{equation}\label{eq:aj2}
 a_j=\sum_{(S_0,S_1,U)} \chi_{G|U}(p), 
\end{equation}
where we sum over all lists $(S_0,S_1,U)$ satisfying these conditions:
$V(G_*)$ is the disjoint union of $S_0$, $S_1$, and $U$;
$|S_0\cup S_1|=j$; the root of $G_*$ is not in $S_0$;
no two vertices in $S_0$ are joined by an edge;  and
no two vertices in $S_1$ are joined by an edge.
We get a proper coloring contributing to the coefficient of $q^j$
in $X_{\neq 0}^{2,p}(G_*)$ by coloring every vertex in $S_0$ with color $0$,
every vertex in $S_1$ with color $1$, and choosing a proper coloring
of $G|U$ using the $p$ available colors $\{2,3,\ldots,p+1\}$.
Formula~\eqref{eq:aj2} follows from the Sum Rule, the Product Rule,
and the definition of $\chi_{G|U}$.

Since $G_*$ is connected and bipartite, there
is exactly one proper coloring of $G_*$ using color set $\{0,1\}$
where the root $r$ of $G_*$ is not colored $0$. In more detail,
for each vertex $s$ of $G$, there exists a path in $G$ from $r$ to $s$
(since $G$ is connected), and the lengths of all paths from $r$ to $s$
have the same parity (since $G$ is bipartite). To obtain a coloring
$\kappa$ with the specified properties, we must set $\kappa(s)=1$
for all $s$ that are an even number of edges away from $r$, and
$\kappa(s)=0$ for all $s$ that are an odd number of edges away from $r$.
So the coefficient of $q^n$ in $X_{\neq 0}^{2,p}(G_*)$ is $1$,
and hence this polynomial is monic of degree $n=|V(G)|$.
(This conclusion also holds for $X_0^{2,p}(G_*)$, by
interchanging the roles of colors $0$ and $1$.)

For $0\leq j<n$, $p$ divides $a_j$ since 
$p$ divides each summand $\chi_{G|U}(p)$ on the right side of~\eqref{eq:aj2}.
(We need $j<n$ to ensure that $G|U$ is a nonempty graph.)
Since $p$ does not divide the coefficient of $x$ in $\chi_G$,
$a_0=\chi_G(p)$ is not divisible by $p^2$.  Eisenstein's Criterion 
now shows that $X_{\neq 0}^{2,p}(G_*)$ is irreducible in $\Q[q]$.
The conclusion for rooted trees follows since every rooted tree $T_*$ 
is non-empty, connected, and bipartite, and the coefficient of $x$
in $\chi_T=x(x-1)^{n-1}$ is $(-1)^{n-1}$.
When $p=2$, the conclusion for any connected bipartite rooted graph
follows from the next lemma. 
\end{proof}

\begin{lemma}
For any graph $G$, the coefficient of $x$ in $\chi_G$ 
is odd if and only if $G$ is connected and bipartite (has no odd cycles).
\end{lemma}
\begin{proof}
Let $L(G)$ be the coefficient of $x$ in $\chi_G$.
We use induction on the number of edges in $G$.
The result is immediate if $G$ has no edges, or if $G$ is not connected
($L(G)$ is zero in the latter case). For the induction step,
assume $G$ is connected with at least one edge.  It follows 
from the deletion-contraction recursion~\cite[Thm.~5.3.6]{west-graph} 
that for any edge $e$ of $G$, $L(G)=L(G-e)-L(G_e)$,
where $G-e$ is obtained from $G$ by deleting edge $e$,
and $G_e$ is obtained from $G$ by contracting edge $e$.
Now consider several cases.

\emph{Case~1.} 
There is an edge $e$ of $G$ whose removal disconnects $G$,
so $L(G-e)=0$. Now $G_e$ is connected, and $G_e$ is bipartite
if and only if $G$ is bipartite, since contracting $e$ does
not create or destroy any cycles.
So $L(G)$ and $L(G_e)$ have the same parity, and the result
follows by induction.
 
\emph{Case~2.}
$G$ is connected and bipartite, but Case~1 does not apply to $G$.
 Let $e$ be any edge of $G$. Then $G-e$ is connected and bipartite,
 while $G_e$ is connected and not bipartite (as $e$ must be part of
 an even-length cycle of G, and contraction makes this an odd cycle).
 By induction, $L(G-e)$ is odd and $L(G_e)$ is even, so $L(G)$ is odd.
 
\emph{Case~3.}
 $G$ is connected with an odd cycle, and $G$ has an edge $e$
 not used in that odd cycle. We can assume Case~1 does not apply to $G$.
 Then $G-e$ is connected with an odd cycle, and
 (as is readily checked) $G_e$ also has an odd cycle.
 By induction, $L(G-e)$ and $L(G_e)$ are even, so $L(G)$ is even.
 
\emph{Case~4.}
None of the first three cases apply to $G$. Then $G$ itself must be an odd
 cycle. For any edge $e$ on this cycle, $G-e$ is a path and $G_e$ is an even 
 cycle (or a single edge), so $G-e$ and $G_e$ are connected and bipartite. 
 By induction, $L(G-e)$ and $L(G_e)$ are odd, so $L(G)$ is even.
\end{proof}

\subsection{Analogue of Stanley's Conjecture for Rooted Trees}
\label{subsec:prove-thm-rooted}

In this section, we prove that for all $N\geq 3$,
the polynomial $X_0(T_*;x_0,\ldots,x_N)$ distinguishes isomorphism
classes of rooted trees. In fact, a much more striking result is true:
the one-variable polynomial $X_0^{2,2}(T_*)=X_0(T_*;q,q,1,1)$ already 
suffices to distinguish rooted trees. 
To prove this, we need a variation of Proposition~\ref{prop:x0-det}.
First we set up notation. For $i\in\{0,1,2,3\}$ and a rooted graph $G_*$,
let $\X_i(G_*)=X_i(G_*;q,q,1,1)$,
let $\X_{\neq i}(G_*)=X_{\neq i}(G_*;q,q,1,1)$, and let $\X_G=X_G(q,q,1,1)$.
Given a polynomial $f=a_nq^n+a_{n-1}q^{n-1}+\cdots+a_1q+a_0\in\Q[q]$ 
of degree at most $n$, let $\rev_n(f)=q^n f(1/q)=a_0q^n+a_1q^{n-1}
 +\cdots+a_{n-1}q+a_n$ be the polynomial obtained by reversing the 
coefficient sequence $(a_0,a_1,\ldots,a_n)$.

\begin{prop}\label{prop:x0-det2}
For any rooted graph $G_*$ with $n$ vertices:
\\ (a) $\X_0(G_*)=\X_1(G_*)$ and $\X_2(G_*)=\X_3(G_*)=\rev_n(\X_0(G_*))$.
\\ (b) $\X_{\neq 0}(G_*)=\X_{\neq 1}(G_*)$ and
       $\X_{\neq 2}(G_*)=\X_{\neq 3}(G_*)=\rev_n(\X_{\neq 0}(G_*))$.
\\ (c) $\X_G=2\X_0(G_*)+2\rev_n(\X_0(G_*))$.
\\ (d) $\X_G=(2/3)\X_{\neq 0}(G_*)+(2/3)\rev_n(\X_{\neq 0}(G_*))$.
\\ (e) $\X_0(G_*)=\X_G-\X_{\neq 0}(G_*)
     =(-1/3)\X_{\neq 0}(G_*)+(2/3)\rev_n(\X_{\neq 0}(G_*))$.
\end{prop}
\begin{proof}
For a permutation $\sigma$ of $\{0,1,2,3\}$
and a coloring $\kappa:V(G_*)\rightarrow \{0,1,2,3\}$,
$\sigma\circ\kappa$ is the coloring obtained from $\kappa$
by replacing each color $c$ by color $\sigma(c)$.
Note $\kappa$ is a proper coloring of $G$ if and only if
 $\sigma\circ\kappa$ is a proper coloring of $G$.
A coloring $\kappa$ contributes to the coefficient of $q^j$ 
in $\X_0(G_*)$ if and only if $(0,1)\circ \kappa$ contributes to 
the coefficient of $q^j$ in $\X_1(G_*)$. So $\X_0(G_*)=\X_1(G_*)$, and
$\X_2(G_*)=\X_3(G_*)$ follows similarly by comparing
$\kappa$ to $(2,3)\circ\kappa$.  On the other hand, a coloring $\kappa$
contributes to the coefficient of $q^j$ in $\X_0(G_*)$ iff
$\kappa'=(0,2)(1,3)\circ\kappa$ contributes to the coefficient of 
$q^{n-j}$ in $\X_2(G_*)$. So $\X_2(G_*)=\rev_n(\X_0(G_*))$. 
This proves (a), and we prove (b) by the same symmetry arguments.
For example, if $\kappa$ colors exactly $j$ vertices using $\{0,1\}$
and does not color the root with color $0$, then $(0,2)(1,3)\circ\kappa$
colors exactly $n-j$ vertices using $\{0,1\}$ and does not color the
root with color $2$ (and conversely).
Since $\X_G=\X_0(G_*)+\X_1(G_*)+\X_2(G_*)+\X_3(G_*)$,
part~(c) follows from part~(a). Similarly,
part~(d) follows from part~(b) and the observation
$\sum_{i=0}^3 \X_{\neq i}(G_*)=\sum_{j=0}^3 3\X_j(G_*)= 3\X_G$. 
The first equality in part~(e) holds by definition, and the second equality
follows from part~(d).
\end{proof}

\begin{theorem}\label{thm:best}
For all rooted trees $T_*$ and $U_*$ and all $N\geq 3$,
the following conditions are equivalent:
(a)~$T_*\cong U_*$; 
(b)~$X_0(T_*;x_0,\ldots,x_N)=X_0(U_*;x_0,\ldots,x_N)$;
(c)~$\X_0(T_*)=\X_0(U_*)$.
\end{theorem}
\begin{proof}
It is routine to check that isomorphic rooted graphs have the same $X_0$ 
polynomial, so (a) implies (b).  Since $N\geq 3$, (b) implies (c) by setting
$x_0=x_1=q$, $x_2=x_3=1$, and $x_k=0$ for $3<k\leq N$.
To prove (c) implies (a), we use induction on $n=\deg(\X_0(T_*))$.
Let $T_*$ and $U_*$ be rooted trees with $\X_0(T_*)=\X_0(U_*)$.
Both $T_*$ and $U_*$ have $n=\deg(\X_0(T_*))$ 
vertices. In the base case $n=1$, we have 
$T_*\cong U_*$ since both trees consist of a single root vertex. 
Assume $n>1$ from now on.
Let the root $r$ of $T_*$ have principal subtrees $T_{1*},\ldots,T_{c*}$, 
and let the root $s$ of $U_*$ have principal subtrees
$U_{1*},\ldots,U_{d*}$. Applying Proposition~\ref{prop:recT}
to $T_*$ and to $U_*$ and specializing at $(q,q,1,1)$, 
the assumption $\X_0(T_*)=\X_0(U_*)$ becomes
\begin{equation}\label{eq:recTU}
 q\prod_{j=1}^c \X_{\neq 0}(T_{j*})=q\prod_{j=1}^d \X_{\neq 0}(U_{j*})
\qquad\mbox{ in $\Q[q]$.}
\end{equation}
Theorem~\ref{thm:better}
shows that each side of~\eqref{eq:recTU} is a factorization
of the polynomial $\X_0(T_*)=\X_0(U_*)$ into monic irreducible factors.
Because $\Q[q]$ is a unique factorization domain, 
we conclude that $c=d$ and (after reordering factors appropriately)
$\X_{\neq 0}(T_{j*})=\X_{\neq 0}(U_{j*})$ for $1\leq j\leq c$. 
For each $j$, we can apply Proposition~\ref{prop:x0-det2}(e),
taking $n$ there to be $\deg(\X_{\neq 0}(T_{j*}))=|V(T_{j*})|=|V(U_{j*})|$,
to deduce that $\X_0(T_{j*})=\X_0(U_{j*})$.
Each principal subtree $T_{j*}$ has fewer than $n$ vertices. 
By the induction hypothesis, $T_{j*}\cong U_{j*}$ for $1\leq j\leq c$. 
Finally, $T_*\cong U_*$ follows from~\eqref{eq:iso-crit}.
\end{proof}

Based on the previous result, one might hope that $f_G(q)=X_G(q,q,1,1)$
suffices to distinguish unrooted trees. This is not true, as seen
in the next example.

\begin{example}
Trees $T$ and $U$ shown below are not isomorphic, but
$f_T=f_U= 2q^{11} + 104q^{10} + 1700q^9 + 11452q^8 + 37804q^7 + 67036q^6 +
67036q^5 + 37804q^4 + 11452q^3 + 1700q^2 + 104q + 2$. 
\[ \includegraphics[width=6in]{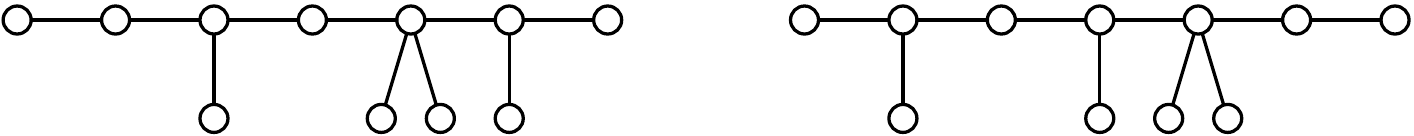} \]
\end{example}

\section{Further Properties of Chromatic Polynomials for Rooted Graphs}
\label{sec:additional}

In this section and the next, certain standard facts regarding various
bases of the ring of symmetric functions $\Lambda$ are
needed. See~\cite{loehr-comb} for background including definitions of
the monomial basis $\{m_{\lambda}\}$, the elementary basis
$\{e_{\lambda}\}$, and the power-sum basis $\{p_{\lambda}\}$.

\subsection{Deletion-Contraction Recursion for $X_0(G_*)$}
\label{sec:delete-contract}

The classical one-variable chromatic polynomials satisfy
the \emph{deletion-contraction recursion} $\chi_G=\chi_{G-e}-\chi_{G_e}$,
where $G$ is a graph, $e$ is an edge of $G$, $G-e$ is the graph $G$
with edge $e$ deleted, and $G_e$ is the graph obtained from $G$
by contracting the edge $e$. There is no such recursion for Stanley's
chromatic symmetric function $X_G$, although 
deletion-contraction recursions are known for the pointed 
chromatic symmetric functions~\cite[Lemma 3.5]{pawlowski}, the
$W$-polynomials and $U$-polynomials~\cite[pg. 1059]{noble-welsh}, and
the noncommutative chromatic functions
$Y_G$~\cite[Prop. 3.5]{gebhard-sagan}. Related recursions can be found
in~\cite{orellana-scott}. We now show
that the rooted version $X_0(G_*)$ does satisfy a simple
generalization of the recursion for $\chi_G$.

\begin{theorem}\label{thm:dc}
Suppose $G_*$ is a rooted graph with root vertex $r$, 
and $e$ is an edge from $r$ to $s$.
Let $G_*-e$ be $G_*$ with edge $e$ deleted (using the same root $r$).
Let $G_{e*}$ be $G_*$ with edge $e$ contracted, meaning that we identify
vertices $r$ and $s$ in $G_e$ and use the identified vertex as the new root.
Then 
\begin{equation}\label{eq:X0-rec}
 X_0(G_*)=X_0(G_*-e)-x_0X_0(G_{e*}). 
\end{equation}
An initial condition occurs when no edge of $G$ touches the root $r$.
In that case, 
\begin{equation}\label{eq:X0-init}
X_0(G_*)=x_0X(G-r;x_0,\ldots,x_N), 
\end{equation}
where $G-r$ is the 
unrooted graph obtained from $G$ by deleting the isolated root vertex.
\end{theorem}
\begin{proof}
Let $S=\Col_0(G_*)$, $T=\Col_0(G_*-e)$, and $U=\Col_0(G_{e*})$.
The set $T$ is the disjoint union of the two subsets
\[ S=\{\kappa\in T: 0=\kappa(r)\neq\kappa(s)\}\quad\mbox{ and }\quad
  U'=\{\kappa\in T: 0=\kappa(r)=\kappa(s)\}. \]
There is a bijection from $U'$ to $U$ sending $\kappa$ to $\overline{\kappa}$,
where the two colorings agree on all vertices other than $r$ and $s$,
and $\overline{\kappa}$ colors the new root vertex $0$.
It follows at once that $\wt(\kappa)=x_0\wt(\overline{\kappa})$.
Since $T$ is the disjoint union of $S$ and $U'$, the Weighted Sum Rule gives
$X_0(G_*-e)=X_0(G_*)+x_0X_0(G_{e*})$, as needed.
The initial condition is immediate from the Weighted Product Rule.
\end{proof}

\begin{example}\label{ex:path3b}
Continuing Example~\ref{ex:path3a}, we use Theorem~\ref{thm:dc} to compute
(for $N=2$)
\begin{align*}
  X_0(\raisebox{-0.1\height}{\includegraphics[width=1cm]{ex-A}}) &=
  X_0(\raisebox{-0.1\height}{\includegraphics[width=1cm]{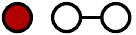}})-x_0X_0(\raisebox{-0.1\height}{\includegraphics[width=0.6cm]{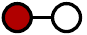}}) 
= x_0(2x_0x_1 + 2x_0x_2 + 2x_1x_2) - x_0(x_0x_1 + x_0x_2)
\\ &= 2x_0x_1x_2 + x_0^2x_1 + x_0^2x_2,
\end{align*}
in agreement with~\eqref{eq:exA}.
\end{example}

Recursions analogous to~\eqref{eq:X0-rec} hold for specializations of 
$X_0(G_*)$.  For example, the polynomials $\X_0(G_*)=X_0(G_*;q,q,1,1)$ 
satisfy $\X_0(G_*)=\X_0(G_*-e)-q\X_0(G_{e*})$
with initial condition $\X_0(G_*)=q\X_{G-r}$.
If we instead use the \emph{principal specialization}
$x_0\rightarrow 1$, $x_1\rightarrow q$, $\ldots$, $x_N\rightarrow q^N$,
then the $x_0$ multiplying the subtracted term in~\eqref{eq:X0-rec} becomes $1$. 
\begin{conjecture}
For all trees $T$ and $U$ with at most $N+1$ vertices,
\begin{equation*}
  T\cong U \text{ if and only if } 
  X_T(1,q,q^2,\ldots,q^N)=X_U(1,q,q^2,\ldots,q^N).
\end{equation*}
\end{conjecture}
We have confirmed this conjecture by computer calculations
for all trees with at most 17 vertices.

\subsection{Interpretation of Coefficients of $x_0^k$ in $X_0(G_*)$}
\label{sec:interpret-coeffs}

Let $G_*$ be a rooted graph. Writing $z=x_0$, we have
\[ X_0(G_*;z,x_1,\ldots,x_N)=\sum_{k\geq 1} c_k(x_1,\ldots,x_N)z^k, \]
where each $c_k(x_1,\ldots,x_N)$ is a symmetric polynomial in $\Lambda_N$.
We now give combinatorial interpretations for these coefficients.
Recall that an \emph{independent set} 
in a graph $G$ is a subset $A$ of $V(G)$ such that no two
vertices in $A$ are joined by an edge in $G$.  Given such an $A$, let
$G-A$ be the graph with vertex set $V(G)-A$ and edge set obtained from
$E(G)$ by deleting all edges with a vertex in $A$ as one endpoint.

\begin{prop}\label{prop:coeff-zk}
Let $G_*$ be a rooted graph.
\\ (a) For $k>0$, let $I(G_*,k)$ be the set of $k$-element 
independent subsets of $G$ that contain the root. Then
\[ X_0(G_*)|_{x_0^k}=\sum_{A\in I(G_*,k)} X_{G-A}(x_1,\ldots,x_N). \]
(b) For $k\geq 0$, let $I'(G_*,k)$ be the set of $k$-element 
independent subsets of $G$ that do not contain the root. Then
\[ X_{\neq 0}(G_*)|_{x_0^k}=\sum_{A\in I'(G_*,k)} X_{G-A}(x_1,\ldots,x_N). \]
\end{prop}
\begin{proof}
We build the proper colorings $\kappa\in\Col_0(G_*)$ contributing
to the coefficient of $x_0^k$ in $X_0(G_*)$ as follows.
First, choose a $k$-element independent set $A\in I(G_*,k)$
and color the vertices of $A$ (including the root $r$) with color $0$.
Second, choose any proper coloring of the graph $G-A$ using
color set $\{1,2,\ldots,N\}$. The formula in (a) follows from the
Sum and Product Rules for Weighted Sets. Part (b) is proved
in the same way, but now we choose $A\in I'(G_*,k)$ to ensure
the root does not receive color $0$.
\end{proof}

\subsection{Reformulation of Stanley's Conjecture}
\label{sec:recovering}

The following proposition illuminates the relationship
between Stanley's original conjecture for unrooted trees 
(Conjecture~\ref{conj:stan}) and the analogous
result for rooted trees (Theorem~\ref{thm:best}).

\begin{prop}\label{prop:recover}
Conjecture~\ref{conj:stan} holds if and only if
for every tree $T$, the multiset $[X_0(T_*^r): r\in V(T)]$
is uniquely determined by the chromatic symmetric function $X_T$.
\end{prop}
\begin{proof}
Assume the condition on multisets is true; we prove Conjecture~\ref{conj:stan}.
Let $T$ and $U$ be any $n$-vertex trees with $X_T=X_U$.  
By the multiset condition, $[X_0(T_*^r): r \in V(T)]=[X_0(U_*^s): s \in V(U)]$.
So there exist $r\in V(T)$ and $s\in V(U)$ with
$X_0(T_*^r)=X_0(U_*^s)$.  By Theorem~\ref{thm:best},
$T_*^r$ and $U_*^s$ are isomorphic as rooted graphs.
So $T$ and $U$ are isomorphic as graphs.

We prove the converse implication by contradiction.
Assume Conjecture~\ref{conj:stan} is true, 
but the condition on multisets is false.
Then there exist trees $T$ and $U$ such that $X_T=X_U$,
but $[X_0(T_*^r):r\in V(T)]\neq [X_0(U_*^s):s\in V(U)]$.
Since $X_T=X_U$, Conjecture~\ref{conj:stan} says $T\cong U$,
so there is a graph isomorphism $f:V(T)\rightarrow V(U)$.
For each $r\in V(T)$, $f$ is a rooted graph isomorphism
from $T_*^r$ to $U_*^{f(r)}$. But then $X_0(T_*^r)=X_0(U_*^{f(r)})$
for all $r\in V(T)$, which means the two multisets are equal.
This gives the required contradiction.
\end{proof}

At present, we do not know how to recover the 
multiset $[X_0(T_*^r):r\in V(T)]$ from $X_T$. 
In fact, this cannot be done for general graphs $G$, 
as the following example shows. This example uses
the \emph{augmented monomial symmetric functions} $\mt_{\lambda}$,
defined by $\mt_{\lambda}=r_1!r_2!\cdots m_{\lambda}$
if $\lambda$ has $r_1$ parts equal to $1$, $r_2$ parts equal to $2$, 
and so on.

\begin{example}
Stanley~\cite[pg. 170]{stanley} gave this example of two non-isomorphic graphs
with equal chromatic symmetric functions:
\begin{equation*}
    X_{\raisebox{-0.4\height}{\includegraphics[width=0.8cm]{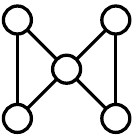}}} =
    X_{\raisebox{-0.4\height}{\includegraphics[width=1.1cm]{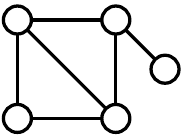}}} =
    2\mt_{221} + 4\mt_{2111} + \mt_{11111}.
\end{equation*}
Let $z=x_0$.  For the first graph $G$, the multiset $[X_0(G_*^r)]$ consists of
 \begin{align*}
    X_0\left(\raisebox{-0.4\height}{\includegraphics[width=0.8cm]{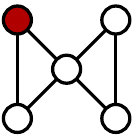}}
\right) =
    X_0\left(\raisebox{-0.4\height}{\includegraphics[width=0.8cm]{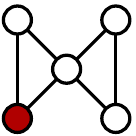}}
\right) =
    X_0\left(\raisebox{-0.4\height}{\includegraphics[width=0.8cm]{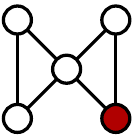}}
\right) =
    X_0\left(\raisebox{-0.4\height}{\includegraphics[width=0.8cm]{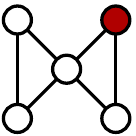}}
\right) &=
    z(2\mt_{211} + \mt_{1111}) + z^2(2\mt_{21} + 2\mt_{111}),\\
    X_0\left(\raisebox{-0.4\height}{\includegraphics[width=0.8cm]{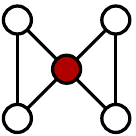}}
\right) &=
    z(2\mt_{22} + 4\mt_{211} + \mt_{1111}).
  \end{align*}
For the second graph $H$, the multiset $[X_0(H_*^s)]$ consists of
  \begin{align*}
    X_0\left(\raisebox{-0.4\height}{\includegraphics[width=1.1cm]{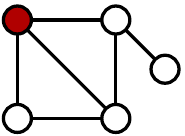}}
\right) =
    X_0\left(\raisebox{-0.4\height}{\includegraphics[width=1.1cm]{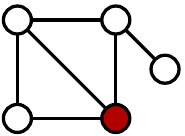}}
\right) &=
    z(\mt_{22} + 3\mt_{211} + \mt_{1111}) + z^2(\mt_{21} + \mt_{111}),\\
    X_0\left(\raisebox{-0.4\height}{\includegraphics[width=1.1cm]{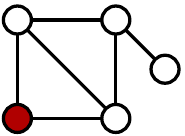}}
\right) &=
    z(2\mt_{211} + \mt_{1111}) + z^2(2\mt_{21} + 2\mt_{111}),\\
    X_0\left(\raisebox{-0.4\height}{\includegraphics[width=1.1cm]{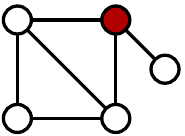}}
\right) &=
    z(3\mt_{211} + \mt_{1111}) + z^2(2\mt_{21} + \mt_{111}),\\
    X_0\left(\raisebox{-0.4\height}{\includegraphics[width=1.1cm]{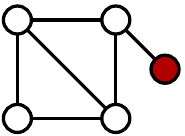}}
\right) &=
    z(\mt_{211} + \mt_{1111}) + z^2(2\mt_{21} + 3\mt_{111}).\\
\end{align*}
This calculation also shows that there exist non-isomorphic rooted graphs
$G_*$ and $H_*$ with $X_0(G_*)=X_0(H_*)$. Indeed, we may take
\[ G_*= \raisebox{-0.4\height}{\includegraphics[width=0.8cm]{stan-A-1}}
\quad\mbox{ and }\quad
   H_*= \raisebox{-0.4\height}{\includegraphics[width=1.1cm]{stan-B-2}}.  \]
\end{example}

Despite this example, there is a simple way
to recover the sum of all $X_0(G_*^r)$ from $X_G$. 

\begin{prop}\label{prop:pointing} For any graph $G$, 
\begin{equation}\label{eq:pointing}
x_0\frac{\partial}{\partial x_0} X_G(x_0,\ldots,x_N)
 = \sum_{r\in V(G)} X_0(G_*^r;x_0,\ldots,x_N).
\end{equation}
\end{prop}
\begin{proof}
The proof is an application of the \emph{pointing construction}
to the generating function $X_G$ (see~\cite[Theorem I.4]{flajolet}).
Each coloring $\kappa\in\Col(G)$ contributes a monomial
$x_0^{e_0}\cdots x_N^{e_N}$ to $X_G$. Applying the operator
$x_0\frac{\partial}{\partial x_0}$ multiplies this monomial by $e_0$,
which is the number of vertices colored $0$ by $\kappa$.
The coloring $\kappa$ belongs to exactly $e_0$ of the
sets $\Col_0(G_*^r)$ as $r$ varies through $V(G)$,
since the root $r$ must be colored $0$. Thus, each coloring $\kappa$
makes the same contribution (namely $e_0x_0^{e_0}\cdots x_N^{e_N}$)
to both sides of~\eqref{eq:pointing}.
\end{proof}

\begin{example}\label{ex:path3c}
Continuing Example~\ref{ex:path3a}, the following
calculation illustrates Proposition~\ref{prop:pointing}:
\begin{align*}
  X_0(\raisebox{-0.1\height}{\includegraphics[width=1cm]{ex-A}}) &+ X_0(\raisebox{-0.1\height}{\includegraphics[width=1cm]{ex-B}}) + X_0(\raisebox{-0.1\height}{\includegraphics[width=1cm]{ex-C}}) \\
  &= (2x_0x_1x_2 + x_0^2(x_1+x_2)) 
+ (2x_0x_1x_2 + x_0(x_1^2 + x_2^2)) + (2x_0x_1x_2 + x_0^2(x_1+x_2))\\
  &= 6x_0x_1x_2 + 2x_0^2(x_1+x_2) + x_0(x_1^2 + x_2^2) = x_0\frac{\partial}{\partial x_0} X_{\includegraphics[width=1cm]{ex}}(x_0,x_1,x_2).
\end{align*}
\end{example}

\subsection{Analogue of the $(\boldsymbol{3}+\boldsymbol{1})$-conjecture}
\label{sec:epos}

Much of the research pertaining to $X_G$ has focused on characterizing
the coefficients of $X_G$ when expressed in various bases of the space
$\Lambda$ of symmetric functions. One outstanding conjecture by
Stanley and Stembridge is the following $e$-positivity conjecture
(see~\cite{AWW,gasharov,pawlowski,shareshian-wachs}
for recent research relevant to this conjecture).

Let $(\mathbf{3}+\mathbf{1})$ be the poset $(\{0,1,2,3\},\preceq)$,
where $1\preceq 2\preceq 3$ and $0$ is incomparable to $1,2,3$.
Given any posets $(P,\leq_P)$ and $(Q,\leq_Q)$, 
we say that $P$ is \emph{$Q$-free} 
if $Q$ is not isomorphic to any induced subposet of $P$. The
\emph{incomparability graph $G(P)$} of a poset $(P,\leq_P)$ 
is the undirected graph with vertex set $P$
and edges $\{x,y\}$ for all $x,y\in P$ such that
$x$ and $y$ are incomparable under $\leq_P$.

\begin{conjecture}{\cite[Conj. 5.1]{stanley}, 
\cite[Conj. 5.5]{stanley-stembridge}}\label{conj:ss}
If $(P,\leq_P)$ is $(\boldsymbol{3}+\boldsymbol{1})$-free,
then $X_{G(P)}$ is $e$-positive.
\end{conjecture}

If we view $X_0(G_*^r;x_0,\ldots,x_N)$ as an element of
$\Lambda[x_0]$, the coefficient of $x_0$ is
$X_{G-r}(x_1,\ldots,x_N)$. For $f\in \Lambda[z]$, we say that $f$
is \emph{$e$-positive} if the coefficient of $z^k$ is an $e$-positive
element of $\Lambda$ for each $k\geq 0$.  
Here is a refinement of Conjecture~\ref{conj:ss},
which we have checked (by computer) for posets with at most $8$ vertices.

\begin{conjecture}
  Let $(P,\leq_P)$ be a poset. For all $r\in P$, if $P-r$ is 
$(\boldsymbol{3}+\boldsymbol{1})$-free, then $X_0(G(P)_*^r)$ is $e$-positive.
\end{conjecture}

When $(P,\leq_P)$ is $(\boldsymbol{3}+\boldsymbol{1})$-free, $G(P)$ is
claw-free (i.e., has no vertex-induced subgraph isomorphic to the
complete bipartite graph $K_{1,3}$). One might hope that we could
strengthen Conjecture~\ref{conj:ss} to assert that $X_0(G_*^r)$ is
$e$-positive whenever $G-r$ is claw-free. However, as discussed
in~\cite{stanley}, claw-free graphs do not have $e$-positive chromatic
symmetric functions in general (see also~\cite{dahlberg}).

\subsection{Power-Sum Formula}
\label{subsec:psum}

Stanley~\cite[Thm. 2.5]{stanley} 
used an inclusion--exclusion argument to find the power-sum
expansion of the chromatic symmetric function $X_G$.  For any edge
$e\in E(G)$, define $A_e$ to be the set of (non-proper) colorings 
$\kappa:V(G)\rightarrow \Z_{\geq 0}$ such that the endpoints of edge $e$ 
receive the same color. For any $S\subseteq E(G)$, 
let $G_S$ denote the graph with vertex set $V(G)$ and edge set $S$.
Let $\lambda(G_S)$ be the integer partition whose parts are the
sizes of the connected components of $G_S$ in weakly decreasing order.
Let $p_{\lambda}$ be the power-sum symmetric function indexed by $\lambda$.
A coloring $\kappa:V(G)\rightarrow \Z_{\geq 0}$ is proper if and only
if it does not belong to the union $\bigcup_{e\in E(G)} A_e$.
Hence, the Inclusion--Exclusion Formula leads to 
\begin{equation}\label{eq:stan-p}
  X_G = \sum_{S\subseteq E(G)} (-1)^{|S|}p_{\lambda(G_S)}.
\end{equation}

We can adapt Stanley's argument to obtain an analogous expansion
of $X_0(G_*;x_0,x_1,\ldots,x_N)$. (This formula is closely related to
the power-sum expansion of Pawlowski's pointed chromatic symmetric
functions, as we explain in Section~\ref{subsec:pawlowski-SF}.)
Given a rooted graph $G_*^v$ with root $v$ and $S\subseteq E(G)$,
let $\lambda_v^+(G_S)$ be the size of the component of $G_S$ containing 
$v$, and let $\lambda_v^-(G_S)$ be the partition $\lambda(G_S)$
with a single part of size $\lambda_v^+(G_S)$ deleted.

\begin{prop}\label{prop:p-exp}
For all rooted graphs $G_*^v$ and $N\geq |V(G)|$,
\begin{equation}\label{eq:p-exp}
 X_0(G_*^v;x_0,x_1,\ldots,x_N)
 =\sum_{S\subseteq E(G)} (-1)^{|S|}
  p_{\lambda_v^-(G_S)}(x_0,x_1,\ldots,x_N)x_0^{\lambda_v^+(G_S)}.
\end{equation}
\end{prop}
\begin{proof}
Let $A^v$ be the set of all (not necessarily
proper) colorings $\kappa:V(G)\rightarrow\{0,1,\ldots,N\}$ where
$\kappa(v)=0$.  For each edge $e$, let $A_e^v$ be the set of
$\kappa\in A^v$ that assign the same color to the two endpoints of
edge $e$.  For each $S\subseteq E(G)$, we can count the weighted set
$\bigcap_{e\in S} A_e^v$ as follows. Write $\lambda_v^+(G_s)=k$, where
$k$ is the size of the component $C^v$ of $G_S$ containing the root
$v$.  Write $\lambda_v^-(G_S)=(\lambda_1,\ldots,\lambda_m)$, where
each $\lambda_i$ is the size of another component $C_i$ of $G_S$.  We
build a coloring in $\bigcap_{e\in S} A_e^v$ as follows.  Color all
$k$ vertices in $C^v$ with color $0$ (as we must), giving a weight
contribution of $x_0^k$.  For each remaining component $C_i$, choose
the common color for all $\lambda_i$ vertices in that component.  The
weight contribution for that choice is
$x_0^{\lambda_i}+x_1^{\lambda_i}+\cdots+x_N^{\lambda_i}
=p_{\lambda_i}(x_0,x_1,\ldots,x_N)$. By the Product Rule, the
generating function for $\bigcap_{e\in S} A_e^v$ is
$p_{\lambda_v^-(G_S)}(x_0,x_1,\ldots,x_N)x_0^{\lambda_v^+(G_S)}$.
Equation~\eqref{eq:p-exp} follows at once from the
Inclusion--Exclusion Formula, since $X_0(G_*^v)$ counts colorings in
$A^v$ outside $\bigcup_{e\in E(G)} A_e^v$.
\end{proof}

\section{Other Variants of Stanley's Chromatic Symmetric Functions}
\label{sec:others}

In this section, we review some of the previously studied polynomials
and algebraic constructs related to the chromatic symmetric function. 
When applicable, we explain the connection between these objects
and the polynomials $X_0(G_*)$.

\subsection{Pawlowski's Pointed Chromatic Symmetric Functions}
\label{subsec:pawlowski-SF}

Given a rooted graph $G_*^v$ with root $v$ and $S\subseteq E(G)$,
Pawlowski's \emph{pointed chromatic symmetric function} is defined
in~\cite[Def. 3.1]{pawlowski} as
\begin{equation}\label{eq:pointed-def}
  P_{G,v} = \sum_{S \subseteq E(G)} 
(-1)^{|S|} p_{\lambda_v^-(G_S)} z^{\lambda^+_v(G_S)-1}\in\Lambda[z],
\end{equation}
where $p_{\lambda}=1$ if $\lambda$ has no positive parts.
The polynomial version using variables $x_1,\ldots,x_N$ is
\begin{equation}\label{eq:pointed-def-5}
  P_{G,v}(x_1,\ldots,x_N) = \sum_{S \subseteq E(G)} 
(-1)^{|S|} p_{\lambda_v^-(G_S)}(x_1,\ldots,x_N) z^{\lambda^+_v(G_S)-1}\in\Lambda_N[z].
\end{equation}

Let $x_0=z$. 
Comparing~\eqref{eq:pointed-def} and~\eqref{eq:p-exp}, we see that
$X_0(G_*^v;z,x_1,\ldots,x_N)$ is the transformation of $P_{G,v}$
obtained by replacing each abstract power-sum $p_k$ by
$z^k+x_1^k+\cdots+x_N^k$ (as opposed to $x_1^k+\cdots+x_N^k$)
and multiplying by $z$. This extra $z$ accounts for
the subtracted $1$ in the exponent of $z$ in~\eqref{eq:pointed-def}.
We state this result more formally as follows.

\begin{prop}\label{prop:X0-vs-paw}
Regard $X_0(G_*^v)$ and $P_{G,v}$ as elements of $\Lambda[z]$,
where $\Lambda=\Q[p_k:k\geq 1]$ is the ring of symmetric functions.
Define evaluation homomorphisms $\phi$, $\psi$ on $\Lambda$ by
setting $\phi(p_k)=z^k+p_k$ and $\psi(p_k)=-z^k+p_k$ for all $k$.
Then $X_0(G_*^v)=\phi(zP_{G,v})$ and $zP_{G,v}=\psi(X_0(G_*^v))$.
\end{prop}

\begin{example}\label{ex:pointed2} For 
$G_*^v=\raisebox{-0.1\height}{\includegraphics[width=1.5cm]{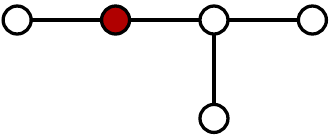}}$, 
  we use Equation~\eqref{eq:pointed-def} to compute
  \begin{equation}\label{eq:pointed-ex-2}
    P_{G,v} = (p_{1111} - 2p_{211} + p_{31})z^0 +
    (-2p_{111}+2p_{21}-p_3)z^1 + (3p_{11})z^2 + (-3p_1)z^3 + z^4.
  \end{equation}
Replacing each $p_k$ by $z^k+p_k$ and simplifying, we obtain
  \begin{equation}
  X_0(G_*^v) = (p_{1111}-2p_{211}+p_{31})z + (2p_{111}-2p_{21})z^2 + p_{11}z^3.
  \end{equation}
  Notice that the coefficient of $z=x_0$ in $X_0(G_*^v)$ and $zP_{G,v}$
  is the same; this holds in general.
\end{example}

Pawlowski proved~\cite[Cor. 3.6]{pawlowski} 
that replacing $z$ by $-z$ in $P_{G,v}$ gives a positive sum of monomials. 
He asked for a combinatorial interpretation of the coefficient of $(-z)^k$
in $P_{G,v}$. We give one such interpretation in Theorem~\ref{thm:ans-paw} 
below. First we need the following definitions. 
Suppose $H$ is a connected graph with $k+1$ vertices. 
For each subset $S$ of $E(H)$, recall $H_S$ is the graph with
vertex set $V(H)$ and edge set $S$.  Define
\[ f(H)=\sum_{\substack{S\subseteq E(H):
      \\ H_S\text{ is connected}}} (-1)^{|S|-k}, \]
which is a signed sum of all connected spanning subgraphs of $H$.
The following combinatorial interpretation of $f(H)$ shows that
$f(H)$ is always a positive integer. 
Fix a total ordering $e_1<e_2<\cdots<e_m$ of the edge set $E(H)$.
Say that a spanning tree $T$ of $H$ is \emph{internal} 
iff for every cycle in $H$ such that all but one edge $e_j$ of the cycle
belongs to $T$, the missing edge is not the largest edge in the cycle.
In Tutte's terminology~\cite[pg. 85]{tutte}, an internal spanning tree
is a spanning tree with no externally active edges.

\begin{lemma}\label{lem:subtree-var}
For any connected graph $H$ with $k+1$ vertices,
$f(H)$ is the number of internal spanning trees of $H$.
\end{lemma}
\begin{proof}
Let $Z=\{S\subseteq E(H): H_S\text{ is connected}\}$,
and let the sign of $S$ be $(-1)^{|S|-k}$ for each $S\in Z$.
It suffices to define a sign-reversing involution $I:Z\rightarrow Z$
whose fixed points are positive and correspond to the internal
spanning trees of $H$. Given $S\in Z$, compute $I(S)$ by the following
algorithm. Process the edges $e_1,e_2,\ldots,e_m$ in the given order.
When edge $e_j$ is reached, do the following.
If $e_j\in S$ and $e_j$ is the largest edge in some cycle of $H_S$,
then return $I(S)=S\setminus\{e_j\}$.
If $e_j\not\in S$ and adding $e_j$ to $S$ creates a cycle with
largest edge $e_j$, then return $I(S)=S\cup\{e_j\}$.
Otherwise, proceed to the next edge. If no output is returned 
after processing all edges, then return $I(S)=S$.

Since $I$ acts by deleting or adding one edge from $S$,
$I$ reverses signs when acting on a non-fixed point.
To see that $I$ is an involution, suppose we compute $I(S)=S'$
by returning an answer when edge $e_j$ is processed.
The decisions made when processing edges $e_1,\ldots,e_{j-1}$
do not depend on whether $e_j$ is or is not in $S$.
So, when we compute $I(S')$, the same decisions will be
made for edges $e_1,\ldots,e_{j-1}$ (namely, continue without
returning an output). When the algorithm processes edge $e_j$, 
it returns $I(S')=S$, as needed.

Now consider the fixed points of $I$.
On one hand, let $S\in Z$ correspond to a spanning subgraph of $H$
that is not a tree. This subgraph has a cycle, so there exists
an edge $e_j$ that is the largest edge in some cycle of $H_S$.
Then $I$ matches $S$ to some $S'$ of opposite sign when processing
$e_j$ (or perhaps earlier). On the other hand, by comparing the
definition of internal spanning tree to the definition of $I$, we see
that the edge set $S$ of a spanning tree of $H$ is a fixed point of $I$
iff this spanning tree is internal. All such trees have $k$ edges
and therefore have positive sign.
\end{proof}

The next theorem shows that the coefficient of $(-z)^k$ in $P_{G,v}$
is the weighted sum of pairs $(T,\kappa)$, where $T$ is an internal
spanning subtree of a $(k+1)$-vertex induced subgraph of $G$ containing the
root $v$, and $\kappa$ is a proper coloring of the vertices of $G$
not in this subgraph.

\begin{theorem}\label{thm:ans-paw}
Let $G$ be a graph with root vertex $v$. For all $k$,
the coefficient of $(-z)^k$ in $P_{G,v}$ is $\sum_{H} f(H)X_{G\setminus H}$,
where we sum over all $H=(V(H),E(H))$ satisfying the following conditions:
$H$ is a connected graph, $|V(H)|=k+1$, $v\in V(H)$,
and $E(H)$ consists of all edges in $E(G)$ joining vertices in $V(H)$.
Here, $G\setminus H$ is the graph obtained from $G$ by deleting all vertices 
in $H$ and all edges incident to those vertices.
\end{theorem}
\begin{proof}
Fix $k$, and let $\mcH$ be the set of all graphs $H$ satisfying the
conditions in the theorem statement. Let $\mcS$ be the set of all
$S\subseteq E(G)$ such that $\lambda_v^+(G_S)-1=k$. By the defining
formula~\eqref{eq:pointed-def}, the coefficient of $(-z)^k$ in $P_{G,v}$ is
\begin{equation}\label{eq:ptd-coeff}
 \sum_{S\in\mcS} (-1)^{|S|-k}p_{\lambda_v^-(G_S)}. 
\end{equation}
Define a function $h:\mcS\rightarrow\mcH$ by letting
$h(S)=H$, where $V(H)$ is the connected component of $v$ in $G_S$.
For each $H\in\mcH$, we can build all $S\in\mcS$ with $h(S)=H$ as follows.
Choose $S_1\subseteq E(H)$ such that $H_{S_1}$ is connected,
choose $S_2\subseteq E(G\setminus H)$ arbitrarily, and let $S=S_1\cup S_2$.
By definition, $p_{\lambda_v^-(G_S)}=p_{\lambda((G\setminus H)_{S_2})}$.
Therefore, equation~\eqref{eq:ptd-coeff} can be rewritten as
\[ \sum_{H\in\mcH}
 \left(\sum_{\substack{ S_1\subseteq E(H):
 \\ H_{S_1} \text{ is connected}}} (-1)^{|S_1|-k}\right)\cdot
\left(\sum_{S_2\subseteq E(G\setminus H)} (-1)^{|S_2|}
 p_{\lambda((G\setminus H)_{S_2})}\right). \]
By Lemma~\ref{lem:subtree-var}, the sum over $S_1$ equals $f(H)$.
By~\eqref{eq:stan-p}, the sum over $S_2$ equals $X_{G\setminus H}$.
\end{proof}

When $G$ itself is a tree, the subgraphs $H$ in Theorem~\ref{thm:ans-paw}
are also trees. In this case, all the multipliers $f(H)$ are $1$.
We see that for trees $T$, the coefficient of $z^k$ in each of
$X_0(T_*^v)$ and $P_{T,v}(-z)$ is a weighted sum of proper colorings
of certain subgraphs of $T$. For $X_0(T_*^v)$, 
these subgraphs are complements of independent sets containing $v$
(Proposition~\ref{prop:coeff-zk}(a)). For $P_{T,v}$, 
these subgraphs are complements of connected subgraphs containing $v$.

\subsection{Rooted $U$-Polynomials}
\label{subsec:rooted-U}

The pointed chromatic symmetric functions are closely related to
another family of polynomials called rooted
$U$-polynomials~\cite{ADZ-rooted}.  To describe these, we first review
the $W$-polynomials and $U$-polynomials of Noble and
Welsh~\cite{noble-welsh}. A \emph{weighted graph} is a pair
$(G,\omega)$ where $G$ is a graph (possibly containing loop edges
or multiple edges with the same endpoints) and
$\omega:V(G)\rightarrow\Z_{>0}$ is a weight function on the vertex set
of $G$. Writing $\bx$ for the sequence of commuting indeterminates
$x_1,x_2,\ldots$, the \emph{$W$-polynomials} $W_{(G,\omega)}(\bx,y)$
are defined recursively using a version of the deletion-contraction
recursion.  As an initial condition, if $G$ has no edges, then
$W_{(G,\omega)}(\bx,y) =\prod_{v\in V(G)} x_{\omega(v)}$.  If $G$ has
a loop edge $e$, then $W_{(G,\omega)}(\bx,y) =
yW_{(G-e,\omega)}(\bx,y)$, where $G-e$ is the graph $G$ with loop $e$
deleted.  If $G$ has an edge $e$ with distinct endpoints $v_i$ and
$v_j$, then
\[W_{(G,\omega)}(\bx,y) = W_{(G-e,\omega)}(\bx,y) + W_{(G_e,\omega_e)}(\bx,y),\]
 where $G_e$ is the contraction of $G$ along $e$
 (i.e., delete $e$ and identify $v_i$ and $v_j$ as a new vertex $v_{ij}$),
 $\omega_e(v_{ij})=\omega(v_i)+\omega(v_j)$, and
 $\omega_e(v)=\omega(v)$ for all $v\not\in\{v_i,v_j\}$.
Note that contraction of an edge could lead to a graph with
a loop edge or multiple edges with the same endpoint.
When $G$ is a forest, the variable $y$ does not appear in $W_{(G,\omega)}$.

\begin{remark}
  In 1994, Chmutov, Duzhin, and Lando~\cite{chmutov} independently
  defined functions on weighted graphs (in the context of Vassiliev
  knot invariants) that are essentially equivalent to the
  $W$-polynomials of Noble and Welsh with $y=0$.  Equation~\eqref{eq:stan-p} 
  is a specialization of an equation appearing at the beginning of
  Section~2.2 in~\cite{chmutov}.
\end{remark}

The \emph{$U$-polynomial} for a graph $G$ is
$U_G(\bx,y)=W_{(G,\mathbf{1})}(\bx,y)$, where
$\mathbf{1}$ is the weight function sending each $v\in V(G)$ to $1$.

\begin{example}
Let $\omega$ be the weight function that assigns $3$ to the leftmost
vertex of $T =
\raisebox{-0.1\height}{\includegraphics[width=1cm]{ex}}$ and $1$ to
the other two vertices. One can check that
  $W_{(T,\omega)}(\bx,y) = x_1^2x_3 + x_2x_3 + x_1x_4 + x_5$
  and $U_T(\bx,y) = x_1^3 + 2x_1x_2 + x_3$. 
\end{example}

The $U$-polynomials can also be defined directly on unweighted graphs $G$
as an alternating sum over subsets of the edge set,
as in~\eqref{eq:stan-p}. Define the \emph{rank} of $S\subseteq E(G)$
to be $r(S) = |V(G)| - k(G_S)$, where $k(G_S)=\ell(\lambda(G_S))$ 
is the number of connected components of $G_S$. Then
\begin{equation}\label{eq:U-altsum}
  U_G(\bx,y) = \sum_{S\subseteq E(G)} \bx_{\lambda(G_S)}(y-1)^{|S|-r(S)},
\end{equation}
where $\bx_{\lambda}=\prod_i x_{\lambda_i}$.
When $G$ is a forest, $k(G_S) = |V(G)|-|S|$, so $r(S)=|S|$ and
the power of $y-1$ disappears. For any loopless graph $G$,
we can recover $X_G$ from $U_G$ by setting $y=0$ and $x_i=-p_i$ for all $i>0$
and multiplying by $(-1)^{|V(G)|}$ (see~\cite[Theorem 6.1]{noble-welsh}).

The \emph{rooted $U$-polynomial}
for the graph $G$ rooted at $v$ is defined in~\cite{ADZ-rooted} as
\begin{equation}\label{eq:RU-altsum}
  U^r(G,v;\bx,y,z) = \sum_{S\subseteq E(G)} 
\bx_{\lambda_v^{-}(G_S)}z^{\lambda_v^+(G_S)}(y-1)^{|S|-r(S)}.
\end{equation}
Note that~\eqref{eq:RU-altsum} modifies~\eqref{eq:U-altsum}
in the same way that~\eqref{eq:pointed-def} modifies~\eqref{eq:stan-p},
up to a factor of $z^{-1}$.
In particular, Pawlowski's pointed 
chromatic symmetric function can be recovered from the rooted
$U$-polynomial by setting $y=0$ and $x_i=-p_i$ for all $i>0$ 
and multiplying by $(-1)^{|V(G)|+1}z^{-1}$.  
(An extra sign is needed since $\lambda_v^-(G_S)$ has one fewer
part than $\lambda(G_S)$.)

\begin{example}
For the rooted graph ${\includegraphics[width=1cm]{ex-B}}$,
we compute $U^r(G,v;\bx,y,z) = x_1^2z+2x_1z^2 + z^3$. 
\end{example}

Corollary~10 of~\cite{ADZ-rooted} proves that the rooted $U$-polynomials
distinguish isomorphism classes of rooted trees. Combining this
with Proposition~\ref{prop:X0-vs-paw} gives a second proof of
the symmetric function version of Theorem~\ref{thm:best}.
In detail, suppose rooted trees $T_*^v$ and $S_*^w$ 
have $X_0(T_*)=X_0(S_*)$.  Then $P_{T,v}=P_{S,w}$
by Proposition~\ref{prop:X0-vs-paw}. Since $y$ does not appear
in the rooted $U$-polynomial for a tree, we can reverse the
transformation above to conclude that $U^r(T,v;\bx,y,z)=U^r(S,w;\bx,y,z)$.
Hence, $T_*$ and $S_*$ are isomorphic by \cite[Cor. 10]{ADZ-rooted}.
Theorem~\ref{thm:best} is a stronger result, showing
that the polynomial version $X_0(T_*;x_0,x_1,x_2,x_3)$
or its specialization $X_0^2(T_*)=X_0(T_*;q,q,1,1)$
already suffices to distinguish rooted trees with any number of vertices.

\subsection{Noncommutative Chromatic Symmetric Functions}
\label{subsec:noncomm}

In 2001, Gebhard and Sagan introduced a noncommutative
version of the chromatic symmetric function~\cite{gebhard-sagan} 
(see also the more recent~\cite{AWW}).
Let $x_1,x_2,\ldots$ be noncommuting indeterminates. Given a graph $G$,
let $v_1<v_2<\cdots<v_n$ be a fixed total ordering of $V(G)$.
The \emph{weight} of a proper coloring $\kappa:V(G)\rightarrow\Z_{>0}$
is $x_{\kappa(v_1)}x_{\kappa(v_2)}\cdots x_{\kappa(v_n)}$.
The \emph{noncommutative chromatic symmetric function}
$Y_G$ is the sum of the weights of all proper colorings. Clearly,
$Y_G$ reduces to $X_G$ when the variables are allowed to commute. 
A primary motivation for introducing $Y_G$ is that $Y_G$ satisfies a
deletion-contraction recursion, which is based on an operation called
induction. Proposition 8.2 of~\cite{gebhard-sagan} shows that two
simple graphs $G$ and $H$ are isomorphic if and only if $Y_G=Y_H$.

\begin{example}
Using a left-to-right ordering of the vertices, we have
$Y_{\raisebox{-0.1\height}{\includegraphics[width=1cm]{ex}}} 
= \sum_{i,j,k} x_ix_jx_k + \sum_{i,j} x_ix_jx_i$, 
where the indices $i,j,k$ in the two sums are required to be distinct.
\end{example}

Let $G_*$ be a rooted graph with root $r$,
and suppose the chosen total ordering $v_1<\cdots<v_n$ of $V(G)$ 
satisfies $v_1=r$.  We can recover $X_0(G_*;x_0,x_1,\ldots,x_N)$ 
from $Y_G(x_0,x_1,\ldots,x_N)$ as follows: discard all monomials
in $Y_G$ that do not begin with $x_0$, and then allow all
indeterminates $x_i$ to commute. 

Each of the functions $X_G$, $X_0(G_*)$, and $Y_G$ reveals certain
information about the proper colorings of $G$. Each monomial in 
$X_G$ tells us the multiset of colors used in the associated coloring 
$\kappa$. The monomials in $X_0(G_*)$ contain similar information, 
but now we are also told that one specific vertex (the root) has color $0$.
The monomial for $\kappa$ in $Y_G$ tells us much more,
since we can recover $\kappa$ itself from the monomial
$x_{\kappa(v_1)}x_{\kappa(v_2)}\cdots x_{\kappa(v_n)}$. For any graph
$G$, it is easy to reconstruct the edge set $E(G)$ from the information 
in $Y_G$~\cite[Prop. 8.2]{gebhard-sagan}. 
Stanley's Conjecture (still open) asks whether a tree $T$
can be reconstructed from the information in $X_T$. Theorem~\ref{thm:best}
shows that a rooted tree $T_*$ can be recovered from $X_0(T_*)$.
This result is tantalizing because $X_0(T_*)$ seems so much closer
to $X_T$ compared to $Y_T$.

\subsection{Chromatic Quasisymmetric Functions}
\label{subsec:qsym-CF}

Shareshian and Wachs~\cite{shareshian-wachs} 
studied another variation of $X_G$ involving
$\QSym$, the ring of quasisymmetric functions in commuting 
indeterminates $x_1,x_2,\ldots$. Fix a total ordering $v_1<v_2<\cdots<v_n$
of $V(G)$. Given a proper coloring $\kappa$ of $G$,
define $\asc(\kappa)$ to be the number of $i<j$ such that
there is an edge in $G$ from $v_i$ to $v_j$ and $\kappa(v_i)<\kappa(v_j)$.
The Shareshian--Wachs \emph{chromatic quasisymmetric function} is
\begin{equation}\label{eq:sw}
  X_G(\bx,t) = \sum_{\kappa} t^{\asc(\kappa)} \wt(\kappa),
\end{equation}
where we sum over all proper colorings 
$\kappa:V(G)\rightarrow\Z_{>0}$. It can be shown
that $X_G(\bx,t)\in \QSym[t]$, meaning that the coefficient of each $t^m$
is a quasisymmetric function.  The specialization $X_G(\bx,1)$ is 
Stanley's symmetric function $X_G$.

\begin{example}
Let $M_\alpha$ be the monomial quasisymmetric function
indexed by the composition $\alpha$. Using a left-to-right
ordering of the vertices, we find
\[ X_{\raisebox{-0.1\height}{\includegraphics[width=1cm]{ex}}}(\bx,t)
 = M_{111} + (M_{21} + M_{12} + 4 M_{111})t + M_{111} t^2. \]
Setting $t=1$ gives $X_G = m_{21} + 6 m_{111}$.
\end{example}

In a related construction, Hasebe and Tsujie~\cite{hasebe-tsujie} define the
\emph{strict order quasisymmetric function} of a poset $P$, as follows.
Let $\Hom^{<}(P,\Z_{\geq 0})$ be the set of functions 
$f:P\rightarrow\Z_{\geq 0}$ such that for all $u,v\in P$,
$u<v$ in $P$ implies $f(u)<f(v)$ in $\Z_{\geq 0}$. Define
\begin{equation*}
 \Gamma^{<}(P,\bx) = \sum_{f\in\Hom^{<}(P,\Z_{\geq 0})} \prod_{v\in P} x_{f(v)}.
\end{equation*}
A rooted tree $T_*$ has an associated poset $P(T_*)$,
where $u\leq v$ in $P(T_*)$ iff the
unique path in $T_*$ from the root to $v$ passes through $u$. 
$\Gamma^{<}(P(T_*),\bx)$ sums over the subset of proper colorings of $T_*$
where the colors strictly increase following any path away from the root.

\begin{example}
We have
$\Gamma^{<}(\raisebox{-0.1\height}{\includegraphics[width=1cm]{ex-B}},
x_0,x_1,x_2) = x_0x_1^2 + 2x_0x_1x_2 + x_0x_2^2 =
X_0(\raisebox{-0.1\height}{\includegraphics[width=1cm]{ex-B}})$.
Equality holds here only because every path from the root has length
$1$ and the number of colors equals the number of vertices.  Note that
$\Gamma^{<}(\raisebox{-0.1\height}{\includegraphics[width=1cm]{ex-A}},
x_0,x_1,x_2) = x_0x_1x_2$, which does not equal
$X_0(\raisebox{-0.1\height}{\includegraphics[width=1cm]{ex-A}};x_0,x_1,x_2)$
as given in~\eqref{eq:exA}.
\end{example}

Hasebe and Tsujie show that $\Gamma^{<}$ is a complete invariant for
rooted trees~\cite[Theorem 1.3]{hasebe-tsujie}. Their proof is
structurally similar to our proof of Theorem~\ref{thm:best}, 
relying on $\QSym$ being a unique factorization domain and on the
irreducibility of various polynomials appearing in the recursion. 
In a related vein, Ellzey~\cite{ellzey2,ellzey} extended the chromatic
symmetric function to directed graphs. In~\cite{McNamara}, Aval,
Djenabou, and McNamara build on this work and generalize some of
the work of Hasebe and Tsujie.

Let us informally compare the information encoded in
$X_0(T_*)$, $X_T(\bx;t)$, and $\Gamma^<(P(T_*))$ for a rooted tree $T_*$
with root $r$.  Assume $X_T(\bx;t)$ is computed using a total ordering 
of $V(G)$ compatible with the poset structure on $T_*$
(so $r$ must be $v_1$). Each monomial in $\Gamma^<(P(T_*))$ records the 
multiset of colors used by a proper coloring, but we also know
that colors must strictly increase along all paths leading away from the root. 
The restriction on which colorings are allowed
for $X_0(T_*)$ is much weaker: we only require
that the root itself receive the smallest available color (color $0$).
$X_T(\bx;t)$ contains even more information than $\Gamma^<(P(T_*))$,
since $\Gamma^<(P(T_*))$ is none other than the coefficient of $t^{|E(G)|}$ 
in $X_T(\bx;t)$ given our assumption on the total ordering of $V(G)$.

We can deduce that $X_T(\bx;t)$ distinguishes rooted trees, in the following
sense.  Suppose $T_*$ and $U_*$ are two $n$-vertex rooted trees with unlabeled 
vertices.  In each tree, label the vertices $1,2,\ldots,n$ in a way that is 
compatible with the poset structure on each tree. Suppose 
$X_T(\bx;t)=X_U(\bx;t)$ using the vertex order $1<2<\cdots<n$.
Then $\Gamma^<(P(T_*))=\Gamma^<(P(U_*))$ by taking the coefficient of $t^{n-1}$
on both sides. By Hasebe and Tsujie's result, $T_*$ is isomorphic to $U_*$.

\subsection{Other Variations for Weighted Graphs}
\label{subsec:other-wtd}

For completeness, we mention a few other variants of $X_G$ for weighted graphs.
In~\cite{crew-spirkl}, the authors 
define extended chromatic functions for weighted graphs
(cf. earlier overlapping results of~\cite{chmutov}).
Given a graph $G$ and a weight function $\omega:V(G)\rightarrow\Z_{>0}$,
define the weight of a proper coloring $\kappa$ of $G$ (relative to
$\omega)$ to be $\prod_{v\in V(G)} x_{\kappa(v)}^{\omega(v)}$.
The \emph{extended chromatic function} $X_{(G,\omega)}$ is the
sum of the weights of all proper colorings of $G$.
These functions satisfy a deletion-contraction recursion and can 
be recovered from the $W$-polynomial. They generalize Stanley's
chromatic symmetric function $X_G$ as well as the 
Shareshian--Wachs chromatic quasisymmetric functions.
In~\cite{alieste}, the authors work with a more
general notion of vertex-weighting and define $M$-polynomials for
marked graphs. They also consider a specialization of
$M$-polynomials called $D$-polynomials. Another generalization,
the $V$-polynomial, is considered in~\cite{ellis-moffatt}.

\section*{Acknowledgments}

The authors acknowledge Sergei Chmutov, Peter McNamara, Rosa Orellana, John
Shareshian, Stephanie van Willigenburg, and Victor Wang for supplying
references and other helpful comments regarding this paper.
The authors extend special thanks to an anonymous referee who showed
us how to generalize our original Theorem~\ref{thm:XG-irred} from
trees to connected graphs.

\end{document}